\documentclass[11pt]{amsart}
\usepackage{epsfig}
\usepackage{amssymb}
\usepackage{amsthm}
\usepackage[all,knot,arc]{xy}
\usepackage{diagrams}
\usepackage{amsmath,amscd}
\usepackage{color}
\usepackage{hyperref}
\hypersetup{colorlinks=true,citecolor=blue,linkcolor=blue}
\newtheorem{theorem}{Theorem}[section]
\newtheorem{lemma}[theorem]{Lemma}

\newtheorem{corollary}[theorem]{Corollary}

 \usepackage{palatino}

\theoremstyle{definition}

\newtheorem{remark}[theorem]{Remark}

\newtheorem{claim}[theorem]{Claim}

\DeclareMathOperator{\tb}{tb}
\DeclareMathOperator{\rot}{r}

\renewcommand{\L}{\ensuremath{\mathcal{L}}}

\numberwithin{equation}{section}

\begin{document}
%opening
\title{Tight small Seifert fibered manifolds with $e_0=-2.$}
%\author{}

\author{B\"{u}lent Tosun}
\address{Department of Mathematics, University of Virginia}
\email{bt5t@virginia.edu}
%\urladdr{http://people.virginia.edu/~bt5t/}

\begin{abstract}

 In this paper we provide the classification of tight contact structures on some small Seifert fibered manifolds. %{$M(-2,r_1,r_2,r_3)$, $r_{i}\in \mathbb{Q}\cap(0,1)$
 As an application of this classification, combined with work of Lekili in \cite{L2010}, we obtain infinitely many counterexamples to a question of Honda-Kazez-Mati\'{c} that asks whether a right-veering, non-destabilizable open book necessarily supports a tight contact structure.
\end{abstract}
\maketitle
\thispagestyle{empty}

%%%%%%%%%%%%%%%%%%%%%%%%%%%%%%%%%%%%%%%%%%%%
 \section{Introduction}
%%%%%%%%%%%%%%%%%%%%%%%%%%%%%%%%%%%%%%%%%%%%

 The classification of tight contact structures on a given closed, oriented $3$--manifold is still one of the main and widely open problems of three dimensional contact geometry. Deep work of Colin, Giroux and Honda in \cite{CGH09} proves that a closed, oriented, atoroidal $3$--manifold has finitely many tight contact structures up to contact isotopy. Conversely, every closed, oriented, irreducible, toroidal $3$--manifold has infinitely many tight contact structures up to contact isotopy. In particular, in the case of closed, oriented Seifert fibered manifolds, if the base orbifold has positive genus or if the number of singular fibers is greater than three (as in both cases the manifold contains a (vertical) incompressible torus), then there are infinitely many tight contact structures on the manifold up to contact isotopy. Finally in \cite{LS09}, the existence question for tight contact structures on the remaining Seifert fibered manifolds was resolved by Lisca and Stipsicz who proved that, except for a small infinite family, all others admit tight contact structures.

%{The classification problem for Seifert fibered manifolds were also extensively studied. We first define the main invariant in the classification of the tight structures. Let $\xi$ be a tight structure on a Seifert fibered manifold and $L$ be a Legendrian knot (smoothly) isotopic to a regular fiber of the Seifert fibration. Let $\phi$ denote this smooth isotopy. Now the twisting number $t(L,\phi)$ is defined as the difference between contact framing and the framing induced by the Seifert fibration. The maximal twisting number of the contact structure $\xi$ is defined to be $t(\xi)=\max_{\phi}\min\{t(L,\phi),0\}$.} 
%{Tight contact structures with negative twisting (see below for the definition) on Seifert fibered manifolds over positive genus base orbifold with any number of singular fibers are classified completely by Massot in (CITE). Tight contact structures on Seifert fibered manifold over genus one orbifold with one cone point are classified completely by Ghiggini in (CITE). When the base orbifold is sphere with two cone points, then the ambient manifold is a lens space and classification was completed by Giroux and independently by Honda(CITE). So, we left with the case of Seifert fibered manifolds over sphere with more than two singular fibers.}  

 In this paper we are interested in the classification problem of tight contact structures on Seifert fibered manifolds over $S^2$ with three singular fibers. Such manifolds are called {\em small} Seifert fibered manifolds and denoted by $M=M(r_{1},r_{2},r_{3})$ where $r_{i}\in \mathbb{Q}$ are the (unnormalized) Seifert invariants that determine $M$ up to orientation and fiber preserving diffeomorphism. Using normalized Seifert invariants we will consider $M(e_0;r_{1},r_{2},r_{3})$ described in Figure~\ref{fig:small} where the {\it integer Euler number} $e_{0}(M)\in \mathbb{Z}$ is an invariant of the Seifert fibration once we require $r_{i}\in \mathbb{Q}\cap(0,1)$. The manifold $M(e_0;r_{1},r_{2},r_{3})$ is an irredicible, atoroidal rational homology sphere, except possibly in the degenerate situation $r_{1}+r_{2}+r_{3}=-e_{0}$: in this case $M$ contains a non-separating, incompressible, horizontal surface, and is therefore a bundle over $S^1$ with periodic monodromy (see~\cite[Proposition~1.11]{Hatcher}). Note that the normalized Seifert invariants clearly satisfy $0<r_1+r_2+r_3<3$, so $M$ can be a surface bundle only when $e_{0}=-2~ \textrm{or}~-1$. Our ability to classify tight contact structures on $M(e_0;r_{1},r_{2},r_{3})$ depends crucially on the value of $r_1+r_2+r_3$.

 The aforementioned result of Colin, Giroux and Honda says unless the small Seifert fibered space is also a torus bundle over the circle, the number of tight contact structures it can admit, up to isotopy, is finite. Hence if $e_{0}\neq -1,~ -2$, then are always finitely many isotopy classes of tight contact structures. The classification of tight contact structures on $M(e_{0};r_{1},r_{2},r_{3})$ for $e_{0}\leq-3$ and $e_{0}>0$ was completed by Wu \cite{W06}. The case $e_{0}=0$ was completed by Ghiggini, Lisca and Stipsicz \cite{GLS06}. The cases $e_{0}=-1,-2$ are incomplete, only some partial results are available due to Ghiggini, Lisca and Stipsicz \cite{GLS07} in case of $e_{0}=-1$ and Ghiggini \cite{G008} in case of $e_{0}=-2$. The purpose of this paper is to study the classification of tight contact structures on $M(-2;r_{1},r_{2},r_{3})$. The classification on such manifolds was initiated in \cite{G008} where Ghiggini among other things gave the classification of tight contact structures on $M(-2;r_{1},r_{2},r_{3})$ which are $L$-spaces. In this paper we generalize his result to some $M(-2;r_{1},r_{2},r_{3})$ that are not $L$-spaces. First we set some notation.

\begin{figure}[h!]
\begin{center}
  \includegraphics[width=5.5cm]{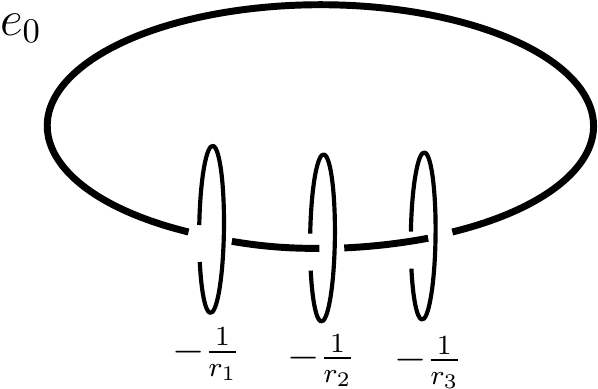}
 \caption{ The manifold $M(e_{0};r_{1},r_{2},r_{3})$. Small unknots in the surgery diagram give rise to singular fibers of Seifert fibration.}
  \label{fig:small}
\end{center}
\end{figure}

\smallskip
 For each of the rational number $r_1,~r_2,~r_3$ in $(0,1)$ we write

\begin{equation}\label{cont}
-\frac{1}{r_i}=[a^{i}_0,a^{i}_1,\cdots,a^{i}_{n_i}]=a^{i}_0 - \cfrac{1}{a^{i}_1 - \cfrac{1}{\ddots \, - \cfrac{1}{a^{i}_{n_i}}}}
\end{equation}

for some uniquely determined integers 

\[
a^{i}_0,a^{i}_1,\cdots,a^{i}_{n_i}\leq-2~,~ ~\qquad~ i=1,2,3
\]

Let $T(r_i)$ denote

\[
T(r_i)=|\prod_{k=0}^{n_{i}}(a^{i}_k+1)|
\]

Recall that we are using the normalized Seifert invariants, that is $r_i \in \mathbb{Q}\cap (0,1)$. We state our first result.   

\begin{theorem}\label{main} If one of the following holds

\begin{enumerate}

\item $r_1+r_2+r_3\geq \frac{9}{4}$

\item $r_1+r_2+r_3 < 2$

\item\label{intsur} $r_1=\frac{1}{2}$, $r_2=\frac{2}{3}$ and $r_3=\frac{k}{k+1}$, for $k\geq 6$.

\end{enumerate}

Then  the manifold $M(-2;r_{1},r_{2},r_{3})$ admits exactly $T(r_1)T(r_2)T(r_3)$ isotopy classes of tight contact structures, all of which are Stein fillable. The explicit Stein filling can be described by Legendrian surgery on all possible Legendrian realizations of the link (after converting each of the $-\frac{1}{r_i}$-framed unknot component to link of $a^{i}_j$- framed unknots with the help of Equation ~\ref{cont} in Figure~\ref{fig:large}). 
\end{theorem}

The third part of Theorem~\ref{main} indicates that the bounds on the sum of the Seifert invariants in previous two are not necessarily sharp. On the other hand, our next result shows that the classification scheme for $M(-2;r_{1},r_{2},r_{3})$  with $2\leq r_1+r_2+r_3<\frac{9}{4}$ is much more complicated.

\begin{theorem}\label{main1}
 The manifold $M(-2;\frac{1}{2},\frac{2}{3},\frac{5n+1}{6n+1})$ for $n\geq 1$,  admits exactly $\frac{n(n+1)}{2}$ isotopy classes of tight contact structures, which are all homotopic and strongly fillable. On the other hand, at least $n$ of them are Stein fillable, and at least $\left\lfloor \frac{n}{2} \right\rfloor$ of the remaining ones are not Stein fillable.    
\end{theorem}

\begin{remark} As mentioned above the manifolds $M(-2;r_{1},r_{2},r_{3})$ with $r_{1}+r_{2}+r_{3} = 2$ are particularly interesting, as they also enjoy a (unique) surface bundle structure over the circle. One can easily determine the fiber genus from the Seifert invariants. For example the manifolds  $M(-2;\frac{1}{2},\frac{3}{4},\frac{3}{4})$, $M(-2;\frac{1}{2},\frac{2}{3},\frac{5}{6})$ and $M(-2;\frac{2}{3},\frac{2}{3},\frac{2}{3})$ are the only torus bundles over circle. Our finiteness result in Theorem~\ref{main} of course cannot include these sporadic cases. We only comment that, the classification of tight contact structures on these manifolds is given independently by Honda in \cite{H00} and Giroux in \cite{G00}. According to that classification each of these manifolds carries infinitely many tight contact structures, distinguished by their Giroux torsion. Among them there is a unique one with Giroux torsion zero, all others have positive Giroux torsion. A result of Gay in \cite{GAY06} proves that positive Giroux torsion is an obstruction to (strong) fillabilty ($\it{cf.}$ \cite{GHM07}), so these manifolds can admit at most one Stein fillable tight structure. On the other hand one can easily see an explicit Stein fillable tight structure from the Figure~\ref{fig:small} on each of these manifolds. The remaining surface bundles are of higher genus and necessarily have periodic monodromy. %As they are both Seifert fibered spaces and small, they cannot have pseudo-Anasov monodromy and do not contain vertical incompressible torus, hence cannot have reducible monodromy.
In particular, they do carry finitely many tight contact structures. Unfortunately we do not have a technique yet to address the classification on such manifolds.
\end{remark}

By using Part~\ref{intsur} of Theorem~\ref{main}, we can extend work of Lekili in ~\cite[Theorem~$1.2$]{L2010} to provide an infinite family of examples of right-veering mapping classes on the four punctured sphere each of which is non-destabilizable and yet supports an overtwisted contact structure. These examples then provide an infinite family of counterexamples for a conjecture of Honda-Kazez-Matic in \cite{HKM2007}. See \cite{L2012}, \cite{KR2012} and \cite{IK2014} for more of such examples.

\begin{corollary}  
For each $k\geq 6$, there are open books $(\Sigma, \phi_k)$ on the Seifert fibered manifolds $M(-2;\frac{1}{2},\frac{2}{3},\frac{k}{k+1})$, where the mapping classes $\phi_k=t^{k+1}_at^2_bt_ct_dt^{-2}_e$ are right-veering, cannot be destabilized and support overtwisted contact structures. See Figure~\ref{Figure2}.    
\end{corollary}

\begin{figure}[!h]
\centering
\includegraphics[scale=0.7]{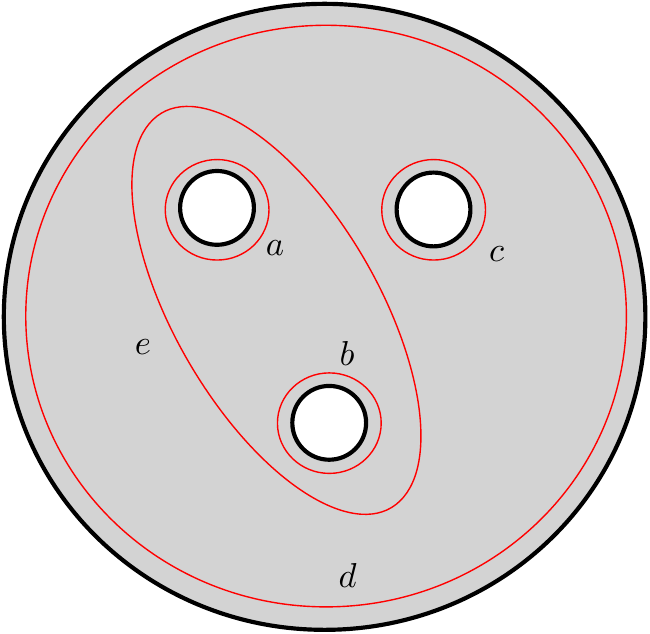} 
\caption{The open book $(\Sigma, \phi_k)$ on the Seifert fibered manifolds $M(-2;\frac{1}{2},\frac{2}{3},\frac{k}{k+1})$.}
\label{Figure2}
\end{figure}

\begin{proof} 
Lekili proves that the mapping classes $\phi_k=t^{k+1}_at^2_bt_ct_dt^{-2}_e$ are right-veering and non-destabilizable. Moreover he proves that \cite[Proposition~$3.2$]{L2010}, the contact invariants $c(\xi_{(\Sigma, \phi_k)})$ of the corresponding contact structures vanish. In particular the contact structures $\xi_{(\Sigma, \phi_k)}$ cannot be Stein fillable. It is not difficult to check the open book smoothly describes the three manifold $M(-2;\frac{1}{2},\frac{2}{3},\frac{k}{k+1})$ on which, by Part~\ref{intsur} of Theorem~\ref{main} above, for each $k\geq 6$, there is exactly one tight contact structure which is Stein fillable.
\end{proof}

\begin{figure}[h!]
\begin{center}
  \includegraphics[width=11cm]{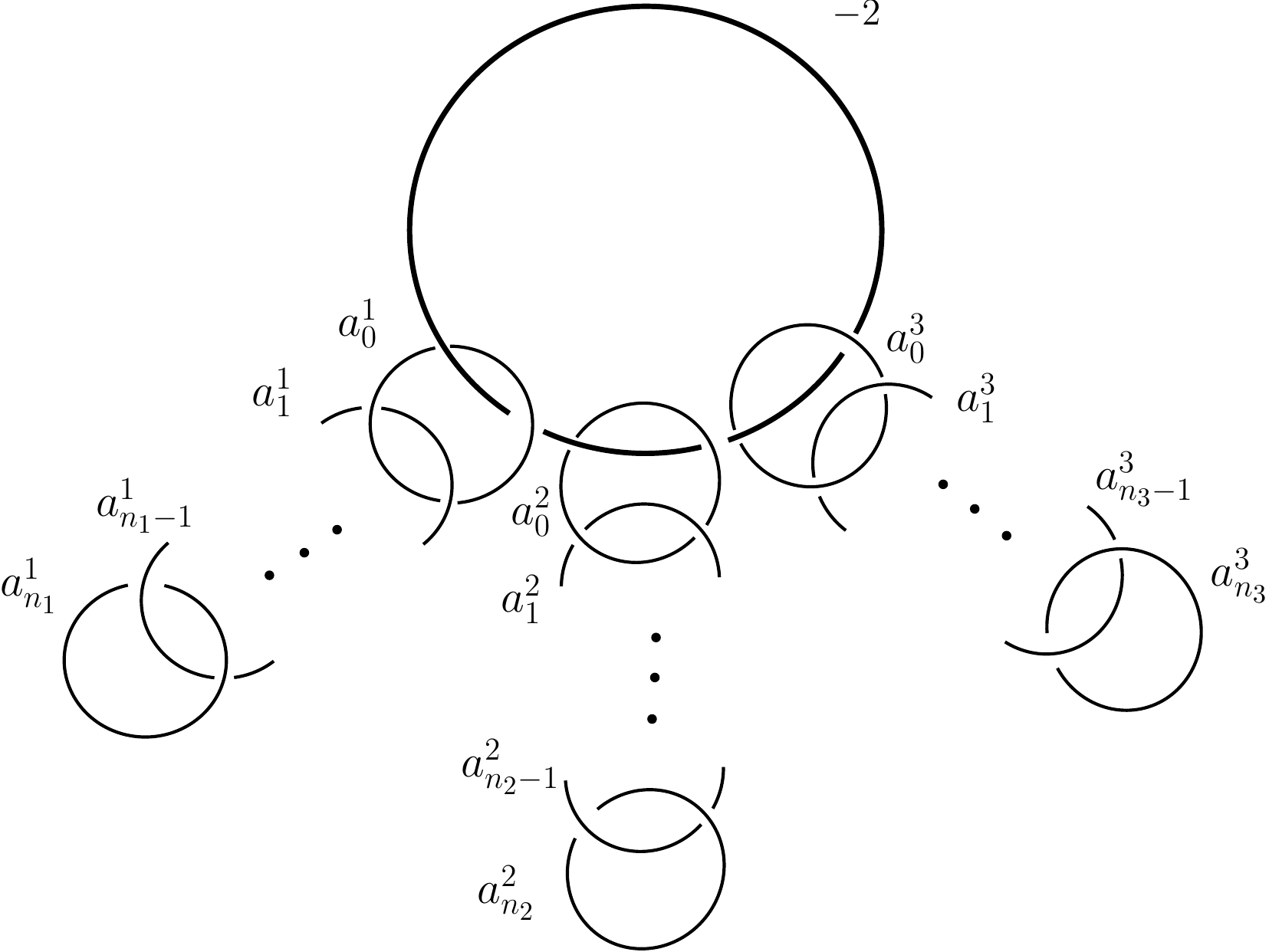}
 \caption{ The manifold $M(-2;r_{1},r_{2},r_{3})$ where $-\frac{1}{r_i}=[a^{i}_0,a^{i}_1,\cdots,a^{i}_{n_i}]$.}
  \label{fig:large}
\end{center}
\end{figure}

\smallskip
\textit {Acknowledgement.}
The author is grateful to Steven Boyer, Adam Clay, John Etnyre, Paolo Ghiggini, Amey Kaloti and  Yank\i\ Lekili  for useful discussions. He would like to particularly thank Tom Mark for very helpful comments on an earlier draft of this work. He was supported in part by an AMS-Simons travel grant, and also thanks the Max Planck Institute for Mathematics for their hospitality
during the summer of 2014.  

%%%%%%%%%%%%%%%%%%%%%%%%%%%%%%%%%%%%%%%%%%%%
\section{Preliminaries}\label{sec:pre}
%%%%%%%%%%%%%%%%%%%%%%%%%%%%%%%%%%%%%%%%%%%%

 In this section, assuming the reader is familiar with convex surface theory of Giroux \cite{G91}, we will list some results regarding bypasses and their consequences due to Honda \cite{H000}. These results will be used again and again in the rest of the paper.
 
We first recall the twisting number which is an invariant in the classification of the tight structures on Seifert fibered manifolds.

Let $\xi$ be a tight structure on a Seifert fibered manifold and $L$ be a Legendrian knot (smoothly) isotopic to a regular fiber of the Seifert fibration. Let $\phi$ denote this smooth isotopy. Now the twisting number $t(L,\phi)$ is defined as the difference between contact framing and the framing induced by the Seifert fibration. The maximal twisting number of the contact structure $\xi$ is defined to be $t(\xi)=\max_{\phi}\min\{t(L,\phi),0\}$. Wu \cite[Theorem~$1.3$]{W06} proved that the maximal twisting number  $t(\xi)<0$ for any tight contact structure $\xi$ on $M(-2;r_{1},r_{2},r_{3})$. The same result was also obtained by Ghiggini \cite[Corollary $4.6$]{G008} and independently  by Massot \cite[Theorem~B]{M2008}.

We now return to describe the effect of bypass attachment in terms of the Farey tesselation of the hyperbolic unit disk.

%\begin{itemize}

\begin{theorem}[The Farey Tessellation ~\cite{H000}] Let $T$ be a convex torus in standard form with $|\Gamma_{T}|=2,$
dividing slope $s$ and ruling slope $r\not=s.$ Let $D$ be a bypass for
$T$ attached to the front of $T$ along a ruling curve. Let $T'$ be the
torus obtained from $T$ by attaching the bypass $D.$ Then
$|\Gamma_{T'}|=2$ and the dividing slope $s'$ of $\Gamma_{T'}$ is
determined as follows: let $[r,s]$ be the arc on $\partial\mathbb{D}$ (where $\mathbb{D}$ is the disc model of the hyperbolic plane)
running from $r$ counterclockwise to $s,$ then $s'$ is the point in
$[r,s]$ closest to $r$ with an edge to $s.$ If the bypass is attached to the back of $T$ then the same algorithm
works except one uses the interval $[s,r]$ on
$\partial\mathbb{D}$. 
\end{theorem}

 As we see that bypasses are useful in changing dividing curves on a
surface and has well understood effect on the dividing curves on a torus. We now mention a standard way to find them, called the
Imbalance Principle.

\begin{theorem}[The Imbalance Principle~\cite{H000}]
 Suppose that $\Sigma$ and $\Sigma'$ are two
disjoint convex surfaces and $\mathcal A$ is a convex annulus whose interior is
disjoint from $\Sigma$ and $\Sigma'$ but its boundary is Legendrian
with one component on each surface. If $|\Gamma_\Sigma\cdot \partial
\mathcal A|>|\Gamma_{\Sigma'}\cdot \partial \mathcal A|$ then there will be a dividing curve on
$\mathcal A$ that cuts a disk off of $\mathcal A$ that has part of its boundary on
$\Sigma$, and hence a bypass for $\Sigma$ on $\mathcal A$.
\end{theorem}

\begin{theorem}[The Twist Number Lemma~\cite{H000}] Consider a Legendrian curve $L$ with twisting number $n$, relative to a fixed framing
and a standard tubular neighborhood V of $L$ with boundary slope $\frac{1}{n}$. If there
exists a bypass $D$ which is attached along a Legendrian ruling curve of slope
$r$, and $\frac{1}{r}\geq n+1$, then there exists a Legendrian curve with larger twisting
number isotopic (but not Legendrian isotopic) to $L$
\end{theorem}

\begin{theorem}[The Edge Rounding Lemma~\cite{H000}] Let $\Sigma_1$ and $\Sigma_2$ be convex surfaces with collared Legendrian boundary which intersect transversely inside the ambient contact manifold along a common boundary Legendrian curve. Assume the neighborhood of the common boundary Legendrian is locally isomorphic to the neighborhood $N_{\epsilon} = \{x^2 + y^2 \leq \epsilon\}$ of $M = \mathbb{R}^2 × (\mathbb{R}/\mathbb{Z})$ with coordinates $(x, y, z)$
and contact $1$–form $\alpha = \sin(2pnz)dx + \cos(2pnz)dy$, for some $n \in \mathbb{Z}^+$, and
that $\Sigma_1 \cap N_{\epsilon} = \{x = 0, 0 \leq y \leq \epsilon\}$ and $\Sigma_2 \cap N_{\epsilon} = \{y = 0, 0 \leq x \leq \epsilon\}.$
If we join $\Sigma_1$ and $\Sigma_2$ along $x = y = 0$ and round the common edge so that the orientations of $\Sigma_1$ and $\Sigma_2$ are compatible and induce the same orientation
after rounding, the resulting
surface is convex, and the dividing curve $z=\frac{k}{2n}$ on $\Sigma_1$ will connect to the dividing curve $z=\frac{k}{2n}-\frac{1}{4n}$
on $\Sigma_2$ , where $k = 0,\cdots,2n-1.$
\end{theorem}
\qed

 %%%%%%%%%%%%%%%%%%%%%%%%%%%%%%%%%%%%%%%%%%%%
 \section{Proof of Theorem~\ref{main}}
 %%%%%%%%%%%%%%%%%%%%%%%%%%%%%%%%%%%%%%%%%%%%

 We first provide some basic facts about continued fractions, set up our framing convention.

\subsection{Continued fractions}

A simple fact for the Farey graph says that, two points on $\partial\mathbb{D}$ correspond to an integral basis of
$\mathbb{Z}^2$ if and only if there is an edge in the Farey tessellation connecting them. Using this fact and a simple induction argument we obtain the following lemma.

\begin{lemma}\label{RelPrime}
Given
\begin{equation}
-\frac{q_i}{p_i}=[a^{i}_0,a^{i}_1,\cdots,a^{i}_{n_i}]=a^{i}_0 - \cfrac{1}{a^{i}_1 - \cfrac{1}{\ddots \, - \cfrac{1}{a^{i}_{n_i}}}}
\end{equation}

for some uniquely determined integers 

\[
a^{i}_0,a^{i}_1,\cdots,a^{i}_{n_i}\leq-2, ~~~~~~ i=1,2,3
\]  

Then 
\begin{enumerate}
\item For each $i=1,~2,~3$,  the numbers $-\frac{v_i}{u_i}=[a^{i}_0,a^{i}_1,\cdots,a^{i}_{n_{i}-1}]$ satisfy $p_{i}\geq u_{i}>0$, $q_{i}\geq v_{i}>0$ and $p_{i}v_{i}-q_{i}u_{i}=1$.
\item $\frac{p_{i}-q_{i}}{v_{i}-u_{i}}=[a^{i}_{n_i},a^{i}_{n_{i}-1},\cdots,a^{i}_{0}+1]$.
\end{enumerate} 
\end{lemma}
\qed
\subsection{Framing}
Let $M=M(\frac{p_1}{q_1}, \frac{p_2}{q_2}, \frac{p_3}{q_3})$ be a Seifert fibered space over $S^2$ with three singular fibers (note that $M\cong M(-2, \frac{p_1}{q_1}, \frac{p_2}{q_2}-1, \frac{p_3}{q_3}-1)$). We choose product framing on $M$: Let $V_{i}$ be a tubular neighborhood of the singular fibers, $F_i$, $i=1,2,3$. Since $M\setminus V_{1}\cup V_{2}\cup V_{3}\cong \Sigma \times S^1$ where $\Sigma$ is a pair of pants. We choose an identification $-\partial(M \setminus V_{i})\cong \mathbb{R}^2/\mathbb{Z}^2$ by setting $(0,1)^{T}$ as the direction of the $S^1$ fiber and $(1,0)^{T}$ as the direction given by $-\partial({pt.} \times \Sigma)$. We choose a different identification  $\partial(V_{i})\cong\mathbb{R}^2/\mathbb{Z}^2$ by setting $(1,0)^{T}$ as the direction of the meridian. Given integers $u_i,v_i$ as in Lemma \ref{RelPrime}, we obtain the Seifert fibered manifold $M\cong(\Sigma \times S^1)\cup_{(A_{1}\cup A_{2}\cup A_{3})} (V_{1}\cup V_{2}\cup V_{3})$ where the attaching maps $A_{i}=\partial V_{i}\rightarrow -\partial(\Sigma\times S^1)_{i}$ are given by

\[
A_i=\left( \begin{array}{cc}
q_i & v_i  \\
-p_i & -u_i \\
\end{array} \right),
~~~ \textrm{with}~ p_iv_i-q_iu_i=1.\]

\bigskip

\begin{proof}[Proof of Theorem \ref{main}] 

Once again, by our framing convention, we consider $M(-2,\frac{p_1}{q_1},\frac{p_2}{q_2},\frac{p_3}{q_3})\cong M(\frac{p_1}{q_1},-\frac{q_2-p_2}{q_2},-\frac{q_3-p_3}{q_3})$ where $\frac{p_i}{q_i}\in \mathbb{Q}\cap(0,1)$ and $0<\frac{p_1}{q_1}\leq\frac{p_2}{q_2}\leq\frac{p_3}{q_3}<1$. 

Let $V_{i}$, for $i=1,2,3$, be standard neighborhoods of the singular fibers $F_i$ with $\textrm{slope}(\Gamma_{\partial V_{i}})=\frac{1}{n_i}$ with $n_i<0$. The attaching maps $A_{i}=\partial V_{i}\rightarrow -\partial(\Sigma\times S^1)_{i}$ are given by 

\[
A_1=\left( \begin{array}{cc}
q_1&v_1\\ -p_1&-u_1
\end{array} \right),\qquad A_2=\left( \begin{array}{cc}
q_2 & v_2 \\
q_2-p_2 & v_2-u_2
\end{array} \right), \qquad A_3=\left( \begin{array}{cc}
q_3 & v_3  \\ q_3-p_3 & v_3-u_3
\end{array} \right)
\]

where $-\frac{q_i}{p_i}=[a^{i}_0,a^{i}_1,\cdots,a^{i}_{n_{i}}]$ and $-\frac{v_i}{u_i}=[a^{i}_0,a^{i}_1,\cdots,a^{i}_{n_{i}-1}]$.  

\bigskip

When measured in $-\partial(M\setminus V_i)$ slopes $s_i,~ i=1,2,3$ are

$s_1=\frac{-p_{1}n_1-u_1}{q_{1}n_{1}+v_1}=-\frac{p_1}{q_1}+\frac{1}{q_{1}(n_{1}q_{1}+v_{1})}$. So, $n_{1}<0$ implies that $-1=\left\lfloor -\frac{p_1}{q_1} \right\rfloor < s_1 <-\frac{p_1}{q_1}$.

For $i=2,3$,  $s_i=\frac{(q_i-p_i)n_{i}+(v_{i}-u_{i})}{q_{i}n_{i}+v_i}$. So, $n_{i}<0$ implies that $0=\left\lfloor -\frac{q_{i}-p_{i}}{q_i} \right\rfloor \leq s_i <-\frac{q_{i}-p_{i}}{q_{i}}$.

In what follows we prove that by finding enough bypasses we can thicken $V_i$'s to have boundary slopes $s_1=-1$ and $s_2=s_3=0$.

First note that after a small isotopy in the neighborhood $V_i$, we can make the ruling curves on $-\partial(M\setminus V_i)$ to have slope infinite, in short these curves will be called vertical. Let $\mathcal{A}$ be an annulus with boundary being Legendrian vertical ruling curves along $V_1$ and $V_2$, such an annulus in short will be called vertical. Note that since $t(\xi)<0$, we can make $\mathcal{A}$ convex. There are two cases.

\vspace{2mm}

{\it Case 1}: If $q_{1}n_{1}+v_1 \neq q_{2}n_{2}+v_2$, the dividing set of $\mathcal{A}$ has boundary parallel arcs on $\partial V_1$ and/or on $\partial V_2$ side. By attaching these bypasses, we can increase the twisting numbers $n_1$, $n_2$ up to $-1$ because of our choice of ruling slopes and the Twist Number Lemma.  

\vspace{2mm}

{\it Case 2}: If $q_{1}n_{1}+v_1 = q_{2}n_{2}+v_2$ and $\mathcal{A} $ has no boundary parallel arcs, then cut along $\mathcal{A}$ and round the corners to get a smooth torus $\partial (M\setminus(V_{1}\cup V_{2}\cup \mathcal{A}))$ that is isotopic to a neighborhood of $F_{3}$. The Edge Rounding Lemma computes the slope as

\[
\begin{split}
s(\Gamma_{\partial(M\setminus V_1 \cup V_2 \cup \mathcal{A})})=-\frac{p_{1}n_1+u_1}{q_{1}n_{1}+v_1} + \frac{(q_2-p_2)n_{2}+v_2-u_2}{q_2n_{2}+v_2} - \frac{1}{q_1n_{1}+v_1}=\\ \frac{(q_1q_2-q_1p_2-p_1q_2)n_1+v_1q_2-v_1p_2-u_1q_2-q_2+1}{q_1q_2n_1+v_1q_2}=\frac{\alpha}{\beta}.
\end{split}
\]    
 
when measured in $\partial V_3$ we get:

\begin{equation}\label{slope}
s_{n_1}=s(\Gamma_{\partial V_3}) = A^{-1}_{3}(-\beta,\alpha)=\frac{(An_1+F)q_3}{(Cn_1+D)v_3}.
\end{equation}

where 

\[
A=\frac{p_1}{q_1}+\frac{p_2}{q_2}+\frac{p_3}{q_3}-2,
\]

\[
 C=2-\frac{p_1}{q_1}-\frac{p_2}{q_2}-\frac{u_3}{v_3},
\]

\[
F=(\frac{p_3}{q_3}+\frac{p_2}{q_2}-2)\frac{v_1}{q_1}+\frac{u_1q_2+q_2-1}{q_1q_2} ~~\textrm{and}~  
D=(2-\frac{p_2}{q_2}-\frac{u_3}{v_3})\frac{v_1}{q_1}-\frac{u_1q_2+q_2-1}{q_1q_2}.
\]
\vspace{2mm}

Note that by our assumption in Theorem~\ref{main} we will consider only the following cases: 

\vspace{1mm}

\begin{enumerate}

\item\label{h} $A\geq \frac{1}{4}$ ( which implies $C<0$) 

\item\label{hh} $A<0$ ( which implies $C>0$).

\item\label{hhh} $\frac{p_1}{q_1}=\frac{1}{2},~ \frac{p_2}{q_2}=\frac{2}{3}$ and $\frac{p_3}{q_3}=\frac{k}{k+1}$, $k \geq 6$

\end{enumerate}

\vspace{1mm}

Before analyzing the cases above in detail we make the following observation which will help us to bypass much of calculations.

\subsection{Key shortcut:}\label{shortcut} Suppose for any $n_1<0$, we find a convex neighborhood of $F_3$ in $V_3$ with slope $\frac{p_3-q_3}{v_3-u_3}$. For notational ease we continue denoting this neighborhood by $V_3$. Now, the slope, when measured in $-\partial (M\setminus V_3)$, becomes $s_3=0$. In particular, there is a Legendrian curve $L$ isotopic to a regular fiber with twisting $-1$, which is the maximal twisting of a tight contact structure $\xi$ on $M$ \cite[Corollary $4.6$]{G008}. We shall use this information to find thickenings of $V_1$ and $V_2$ so that their boundary slopes are $s_1=-1$ and $s_2=0$. We first put a vertical annulus $\mathcal A$ between $V_1$ and $V_3$ such that $\partial \mathcal A=F_1\cup L$. Note that $\Gamma_{\mathcal A}$ cannot have boundary curves on the $V_3$ side. So, since $|q_1n_1+v_1|>1$ whenever $n_1<-1$, there are boundary parallel curves-- bypasses-- in the dividing set of $A_1$ on the $V_1$ side. By attaching these bypasses, we can increase the twisting number up to $n_1=-1$. If we still have $|v_1-q_1|>1$, then there are more bypasses. We keep continue attaching these bypasses until $s_1=-1$ or $|v_1-q_1|=1$. The latter possibility implies that $|p_1-u_1|=1$, as $0<p_1<q_1$,~ $0<u_1<v_1$ and $p_1v_1-u_1q_1=1$. So, in either case we can thicken $V_1$ so that $s_1=-1$. Almost the identical argument shows that we can thicken $V_2$ so that  $s_2=0$. Hence we conclude that $M(\frac{p_1}{q_1}, \frac{q_2-p_2}{q_2},\frac{q_3-p_3}{q_3})=\Sigma \times S^1\cup_{A_1\cup A_2\cup A_3}(V_1\cup V_2 \cup V_3)$ with $s_1=-1$ and $s_2=s_3=0$ measuring in $-\partial(M\setminus V_{i})$,~ $i=1,2,3$. We will now count total possible number of tight contact structures, up to isotopy, on $M$ from each of the pieces. By \cite[Lemma 5.1--3b]{H00}  there is a unique tight contact structure on $\Sigma \times S^1$. By Lemma~\ref{RelPrime}-$(2)$ the slope of $\Gamma_{\partial V_1}$ is  $\bigl(\begin{smallmatrix}
-u_1&-v_1\\ p_1&q_1
\end{smallmatrix} \bigr)\bigl(\begin{smallmatrix}
1\\-1
\end{smallmatrix} \bigr)=\frac{p_1-q_1}{v_1-u_1}=[a^{1}_{n_1},a^{1}_{n_1-1},\cdots,a^{1}_1,a^{1}_{0}+1]$. Thus,  by the classification of tight contact structures on the solid tori \cite[Theorem $4.16$]{H000}, there are exactly $|(a^{1}_0+1)(a^{1}_1+1) \cdots (a^{1}_{n_{1}+1})|=T(r_1)$ tight contact structures on $V_1$. Similarly on $V_2$ and $V_3$ (as $\bigl(\begin{smallmatrix}
v_i-u_i&-v_i\\ p_i-q_i&q_i
\end{smallmatrix} \bigr)\bigl(\begin{smallmatrix}
1\\0
\end{smallmatrix} \bigr)=\frac{p_i-q_i}{v_i-u_i}=[a^{i}_{n_i},a^{i}_{n_i-1},\cdots,a^{i}_1,a^{i}_{0}+1]$; $i=2,3$), there are exactly $T(r_2)$ and $T(r_3)$ tight contact structures, respectively. Hence, under the assumption that there is a neighborhood $V_3$ of $F_3$ with slope $\frac{p_3-q_3}{v_3-u_3}$, we obtain an upper bound on the number of tight contact structure on $M$. Finally this upper bound is achieved by counting all distinct Stein fillable structures on $M$ \cite[Theorem~$1.2$]{LM97} which are obtained by realizing the diagram in Figure~\ref{fig:large} by Legendrian surgeries. 

\vspace{2mm}

\vspace{2mm}

\begin{proof}[Proof of \ref{h} and \ref{hh}] Note that for each of the cases \ref{h} and \ref{hh} we have $s_{n_1}$ is increasing and limits to $\frac{Aq_3}{Cv_3}<-1$ as $n_1\rightarrow -\infty$. Moreover $\frac{Aq_3}{Cv_3}=\frac{p_3-(2-\frac{p_1}{q_1}-\frac{p_2}{q_2})q_3}{(2-\frac{p_1}{q_1}-\frac{p_2}{q_2})v_3-u_3}\leq\frac{p_3-q_3}{v_3-u_3}$ if and only if either $\frac{p_1}{q_1}+\frac{p_2}{q_2}\leq1$ or $A=\frac{p_1}{q_1}+\frac{p_2}{q_2}+\frac{p_3}{q_3}-2>0$ and $C=2-\frac{p_1}{q_1}-\frac{p_2}{q_2}-\frac{u_3}{v_3}<0$. So, by using this inequality and Theorem $4.16$ of \cite{H000} we find the desired neighborhood $V_3$, and hence prove Theorem~\ref{main}, for the case of \ref{h} and of \ref{hh} when $\frac{p_1}{q_1}+\frac{p_2}{q_2}\leq1$. We want to note that it is easy to derive \cite[Theorem~$1.1$]{LS07} that for $M(-2, \frac{p_1}{q_1}, \frac{p_2}{q_2}, \frac{p_3}{q_3})$ to be an $L$--space, it is necessary that $\frac{p_1}{q_1}+ \frac{p_2}{q_2}+ \frac{p_3}{q_3}<2$, and sufficient that $\frac{p_1}{q_1}+ \frac{p_2}{q_2}+ \frac{p_3}{q_3}<2$. In particular, the condition that $\frac{p_1}{q_1}+\frac{p_2}{q_2}\leq1$ implies that the underlying three manifolds $M(-2,\frac{p_1}{q_1},\frac{p_2}{q_2},\frac{p_3}{q_3})$ are $L$-spaces and reproves the classification of tight structures on such manifolds which was first obtained by Ghiggini in \cite{G008}. But as noted above this condition does not characterize all $M(-2, \frac{p_1}{q_1}, \frac{p_2}{q_2}, \frac{p_3}{q_3})$ which are $L$-spaces. 

So it is left to analyze the case of \ref{hh} when $\frac{p_1}{q_1}+\frac{p_2}{q_2}>1$. We first give an explicit demonstration of the classification of tight contact structures on a particular family of small Seifert fibered spaces which are not $L$-spaces: $M_{r}=M(-2,\frac{1}{2},\frac{2}{3},r)$ where $r=\frac{p}{q}\in\mathbb{Q}\cap [\frac{4}{5},\frac{5}{6})$ and $-\frac{q}{p}=[a_1,a_2,\cdots,a_{m}]$.

\bigskip

\begin{lemma}\label{guidelemma}
On $M_{r}$, up to isotopy, there are exactly $|(a_{1}+1)(a_{2}+1)\cdots(a_{m}+1)|$ tight contact structures, which are all Stein fillable.
\end{lemma}  
     
We remark that the lemma, in particular, says that on $M_{n}=M(-2,\frac{1}{2},\frac{2}{3},\frac{5n-1}{6n-1})$, $n\geq2$ (note $-\frac{6n-1}{5n-1}=[-2,-2,-2,-2,-3,-2,\cdots,-2]$, where $-3$ is followed by $(n-2)$,  $-2$'s) there are, up to isotopy, exactly two tight contact structures, which are both Stein fillable. It is interesting to compare this infinite family with its orientation reversal one: $-M_{n}=M(-1,\frac{1}{2},\frac{1}{3},\frac{n}{6n-1})$, which has been very instrumental in our understanding of three dimensional contact geometry, by work of Etnyre and Honda \cite{EH01}, Lisca-Mati\'{c} \cite{LM97} and Ghiggini \cite{G08}. Finally, recent work of Ghiggini and Van Horn-Morris \cite{GM09}, shows that on $-M_{n}$ there are, up to isotopy, exactly $\frac{n(n-1)}{2}$ tight contact structures, which are all strongly fillable but at least one of them is not Stein fillable.

\begin{proof}
Let $\xi$ be a tight contact structure on $M(-2,\frac{1}{2},\frac{2}{3},\frac{p}{q})\cong M(\frac{1}{2},-\frac{1}{3},-\frac{q-p}{q})$. 

When measured in $-\partial(M\setminus V_i)$ slopes are
\[
s_1=-\frac{n_1}{2n_{1}+1},~ s_2=\frac{n_{2}+1}{3n_{2}+2},~~s_3=\frac{(q-p)n_{3}+v-u}{qn_{3}+v}.
\]

After cutting and rounding a standard vertical annulus between $V_1$ and $V_2$ as explained at the beginning above, we get a neighborhood $V_3$ of $F_3$ such that its slope, by Formula~\ref{slope} is   

\[
s_{n_1}=s(\Gamma_{\partial V_3})=\frac{(6p-5q)n_{1}+3p-2q}{(5v-6u)n_{1}+2v-3u}.
\]

This is the place that it is not necessarily true that $s_{n_1}\leq \frac{p-q}{v-u}$ for all $n_1<0$. That is, it is not immediate that there is a thickening of $V_3$ such that its slope when measured with respect to $-\partial(M_{r}\setminus V_3)$ is $s_3=0$. But since $\frac{p}{q}\in \mathbb{Q}\cap[\frac{4}{5}, \frac{5}{6})$, $p\geq u>0$, $q\geq v>0$ and $q-p \geq v-u$, we get that $\frac{(6p-5q)n_{1}+3p-2q}{(5v-6u)n_{1}+2v-3u}<\frac{6p-5q}{5v-6u}$ for all $n_1<0$. So, by \cite[Theorem $4.16$]{H000}, we can find a convex neighborhood $V'_{3}\subset V_{3}$ of the singular fiber $F_3$ such that $s(\Gamma_{\partial V'_{3}})=\frac{6p-5q}{5v-6u}$. This slope becomes $\frac{1}{6}$ when measured in coordinates of $-\partial(M\setminus V'_3)$. We now take a vertical annulus between $V_{1}$ and $V'_{3}$. Note that as long as $n_{1}\leq -4$ we have $|2n_{1}+1|>6$ and hence the convex annulus will have a bypass on $V_{1}$ side. By the Twist Number Lemma attaching this bypass will increase the twisting $n_{1}$ to $-3$. Similarly we can increase the twisting $n_{2}$ to $-2$. So, slopes in the coordinates of $-\partial(M\setminus V_i)$ become $s_1=-\frac{3}{5}, s_{2}=\frac{1}{4}$ and $s_{3}=\frac{1}{6}$. Observe that yet another vertical annulus $\mathcal A$ between $V_{1}$ and $V_{2}$ will result in bypasses which allow us to thicken $V_1$ and $V_2$ so that their boundary slopes become $s_1=-1$ and $s_2=0$. At this point we can assume that the vertical annulus $\mathcal A$ has no boundary parallel arcs in its dividing set because $t(\xi)<0$ for any tight structure on $M_{r}$. We cut along $\mathcal A$ and round the corners of $(M\setminus V_1 \cup V_2 \cup\mathcal A)$. Now $\partial(M\setminus V_1 \cup V_2 \cup \mathcal A)$ is smoothly isotopic to $\partial(M\setminus V_3)$ and has slope, by the Edge Rounding Lemma, $0$. Since the solid tori $V_1$ and $V_2$ have boundary slopes $A^{-1}_{1}(-1,1)=-1$ and $A^{-1}_{2}(1,0)=\infty$, respectively, they are standard neighborhoods of Legendrian (singular) fibers. So each carries a unique tight contact structure. On the other hand, since $A^{-1}_{3}(1,0)=\frac{p-q}{v-u}$ and, by the second part of Lemma \ref{RelPrime} $\frac{p-q}{v-u}=[a_{m},a_{m-1},\cdots,a_2, a_{1}+1]$, we conclude, by using the classification of tight contact structures on solid tori \cite[Theorem $2.3$]{H000} that $V_3$ admits exactly $|(a_{m}+1)(a_{m-1}+1)\cdots(a_2+1)(a_{1}+1)|$ tight contact structures. Therefore, as explined in \ref{shortcut}, $M_{r}$ admits exactly $|(a_{m}+1)(a_{m-1}+1)\cdots(a_2+1)(a_{1}+1)|$ tight contact structures, up to isotopy which are all Stein fillable.      
\end{proof}

 We now return the general case. Note that $s_{n_1}<\frac{Aq_3}{Cv_3}$ for every $n_1<0$, so there is a convex neighborhood $V'_3$ of $F_3$ such that $s(\Gamma_{\partial V'_{3}})=\frac{Aq_3}{Cv_3}$. Now the slope, when measured in $-\partial (M\setminus V'_3)$, becomes $s_3=\bigl(\begin{smallmatrix}
q_3&v_3\\ q_3-p_3& v_3-u_3
\end{smallmatrix} \bigr)\bigl(\begin{smallmatrix}
Cv_3\\Aq_3
\end{smallmatrix} \bigr)=\frac{p_1}{q_1}+\frac{p_2}{q_2}-1>0$. 

{\it Step I:---} Assume first that $q_1=q_2=q$, we then have $s_1=\frac{-p_1n_1-u_1}{qn_{1}+v_1}$, $s_2=\frac{(q-p_2)n_{2}+(v-u_2)}{qn_{2}+v_2}$ and $s_3=\frac{p_1+p_2-q}{q}$. Put a vertical annulus between $V_1$ and $V'_3$. One can easily see that $|qn_1+v|>q$ whenever $n_1<-1$. So by the Imbalance Principle there are bypasses at $V_1$ side. By the Twist Number Lemma from \cite{H000} we can increase the twisting number $n_1$ up to $-1$. Similarly we can increase the twisting number $n_2$ up to $-1$. So with this thickening at hand slopes become $s_1=\frac{p_1-u_1}{v_1-q}$ and $s_2=\frac{p_2-q+v-u_2}{v_2-q}$. Now a vertical annulus between $V_1$ and $V_2$ will have boundary parallel arcs if $|v_1-q|\neq |v_2-q|<$. From which we either obtain that $s_1=-1$ and $s_2=0$ or $v_1=v_2$. If the former happens then we are done by \ref{shortcut}. So we can assume without loss of generality that $v_1=v_2=v$. This time a vertical annulus between $V_1$ and $V'_3$, as $|v-q|<q$, by the Imbalance Principle, must have bypasses along $V'_3$ side. Attaching those bypasses and tracing their effect via the Farey Tesellation we eventually get first that $s_3=\frac{1}{q'}$ and then a further bypass (which is still available) gives that $s_3=0$. Now using the the same argument as in \ref{shortcut}, we obtain the desired classification.

%\vspace{2mm}

{\it Step II:---} Assume now that $q_1 \neq q_2$. So, $s_3=\frac{p_1q_2+p_2q_1-q_1q_2}{q_1q_2}>0$. We now put a vertical annulus between $V_1$ and $V'_3$. Since $|q_1n_1+v_1|>q_1q_2$ whenever $n_1<-q_2$,  by the Imbalance Principle there are bypasses at $V_1$ side which by Twist Number Lemma from \cite{H000} increase the twisting number $n_1$ up to $-q_2$. Similarly we can increase the twisting number $n_2$ up to $-q_1$ and slopes become $s_1=\frac{p_1q_2+u_1}{q_1q_2-v_1}$ and $s_2=\frac{q_1q_2-q_1p_2+u_2-v_2}{q_1q_2-v_2}$. If $v_1\neq v_2$, then $|v_1-q_1q_2|\neq|v_2-q_1q_2|$. So a vertical annulus between $V_1$ and $V_2$ will have bypasses. Once again by the Twist Number Lemma we increase the twisting numbers $n_1$, $n_2$ up to $-q_2+l$, $-q_1+k$ for $1 \leq l\leq q_2-1$ and $1\leq k\leq q_1-1$. As in the case above, from this we either obtain that $s_1=-1$ and $s_2=0$, and be done with the proof or $v_1=v_2=v$. In the latter case we have $|v-q_1q_2|=|v-q_1q_2|\neq q_1q_2$. Hence a vertical annulus between $V_1$ (or $V_2$) and $V'_3$ will produce bypasses on $V'_3$ side that changes its slopes, again referring to the Farey Tesellation, to sequence of fractions ending at $\frac{1}{q'}$. Since $|v-q_1q_2|\neq q'$ is still the case, we get one more bypass, and conclude that $s_3=0$. This finishes the proof.

\end{proof}

\begin{proof}[Proof of ~\ref{hhh}] 
We have already done most of the work at the beginning of the proof. We substitute $\frac{p_1}{q_1}=\frac{1}{2}, \frac{p_2}{q_2}=\frac{2}{3}$ and $\frac{p_3}{q_3}=\frac{k}{k+1}$ in there to obtain that the slopes $s_i,~ i=1,2,3$, measured with respect to $-\partial(M\setminus V_i)$, are $s_1=-\frac{n_1}{2n_1+1}, ~ s_2=\frac{n_2+1}{3n_2+2}$ and $s_3=\frac{n_3+1}{(k+1)n_3+k}$. Moreover,  the new boundary slope, after cutting and rounding a standard vertical annulus between $V_1$ and $V_2$, for $\partial V_3$, by slope Formula ~\ref{slope}, is $s_{n_1}=-\frac{(k-5)n_1+k-2}{(k-6)n_1+k-3}$. Recall we assumed that $k\geq 6$. There are few cases to consider. If $k\geq 8$, then it is easy to see that $s_{n_1}\leq -1$ for all $n_1<0$. In particular there exist a neighborhood $V'_3\subset V_3$ with boundary slope $s(\Gamma_{\partial V'_3})=-1$. When measured with respect to $-\partial(M\setminus V'_3)$, the slope becomes $0$. Now by the argument in \ref{shortcut} (note $\frac{p_3-q_3}{v_3-u_3}$=-1), we finish the proof. When $k=6$ or $7$, one need more care. For example it is not true that we can find a convex neighborhood of $F_3$ with the boundary slope $-1$ for all $n_1<0$. In some sense this is primary reason that we cannot extend our classification to include all the values of the Siefert invariants $\frac{p_i}{q_i}$. Indeed, in the next section we will exhibit an infinite family of examples for which the classification is very different than what we have in Theorem~\ref{main}. Nevertheless, for $k=6,7$, it is easy to check that, similar to the arguments as in Lemma~\ref{guidelemma}, the classification can be obtained as claimed (indeed $k=6$ corresponds to $n=1$ in Theorem~\ref{main1}).

\end{proof}

\end{proof}

%%%%%%%%%%%%%%%%%%%%%%%%%%%%%%%%%%%%%%
\section{Proof of Theorem~\ref{main1}}
%%%%%%%%%%%%%%%%%%%%%%%%%%%%%%%%%%%%%%

Let $M_n$ denote the manifold $M(-2;\frac{1}{2}, \frac{2}{3},\frac{5n+1}{6n+1})$. Our proof of the fact that $M_n$ carries exactly $\frac{n(n+1)}{2}$ tight contact structures, up to isotopy, for any $n\geq1$ is very much parallel to beautiful work of Ghiggini and Van Horn-Morris in \cite{GM09}. We will essentially explain main points and remark necessary changes to conclude our proof. The proof starts with explaining why is that $M_n$ has at most $\frac{n(n+1)}{2}$ tight contact structures. This first part is classic convex surface theory argument as in \cite{EH01}, \cite{W06}, \cite{G008}, \cite{GLS06}, \cite{GLS07} and \cite{GM09}. The second part is devoted to detect the claimed number of tight contact structures which involves two important ideas; Heegaard Floer homology with twisted coefficients and open book decompositions. We refer the reader \cite[Section $3$]{GM09} for Heegaard Floer homology with twisted coefficients basics. We still assume the terminology from the previous section.

\subsection{Upper Bound.} We prove that $M_{n}$ admits at most $\frac{n(n+1)}{2}$ tight contact structures, up to isotopy, for any $n\geq1$. The manifold $M_n$ can also be described as $M(\frac{1}{2}, -\frac{1}{3}, -\frac{n}{6n+1})$. In particular, we can decompose the manifold $M_n$ as

\[
M_n\cong(\Sigma \times S^1)\cup_{(A_{1}\cup A_{2}\cup A_{3})} (V_{1}\cup V_{2}\cup V_{3})
\]
 where the attaching maps $A_{i}=\partial V_{i}\rightarrow -\partial(\Sigma\times S^1)_{i}$ are given by

\[
A_1=\left( \begin{array}{cc}
2&1\\ -1&0
\end{array} \right), \qquad A_2=\left( \begin{array}{cc}
3 & 2 \\1 & 1
\end{array} \right), \qquad A_3=\left( \begin{array}{cc}
6n+1 & 6  \\ n & 1
\end{array} \right);
\]

\begin{lemma}\label{twisting}   
The manifold $M_n$ admits at most $\frac{n(n+1)}{2}$ tight contact structures up to isotopy for any $n\geq1$. 
\end{lemma}

\begin{proof}

 We would like to first determine the maximal twisting number of tight contact structures on $M_n$. It was proven first by Wu in ~\cite{W06} that for any tight contact structure $\xi$ on small Seifert fibered spaces with $e_0=-2$, its maximal twisting number $tw(\xi)<0$. In particular this fact applies to the tight contact structures on $M_n$.

\begin{claim}
If $\xi$ is a tight contact structure on $M_n$. Then the maximal twisting number $tw(\xi)=-6k-1$ for some $k$ with $0\leq k\leq n-1$
\end{claim}

\begin{proof} As practiced in the previous sections, we want to use a vertical annulus between the standard neighborhoods $V_1$ or $V_2$  and Legendrian regular fiber $L$ with the maximal twisting number to produce bypasses on $V_1$ or $V_2$ side, and hence normalize/thicken these neighborhoods as much as we can. Then, based on their slopes determine the twisting number of potential tight contact structures that permit those slopes. 

Let $tw(\xi)=-t<0$, and $L$ be a Legendrian regular fiber with $tw(L,\xi)=-t$. After some small isotopy, we can arrange the decomposition above so that the neighborhoods $V_i$ of the singular fibers $F_i$ and $L$ do not intersect. The slopes of the standard neighborhoods are $-\frac{1}{m_i}$. When measured in $-\partial(M_n\setminus V_i)$ they become

\[
s_1=-\frac{m_1}{2m_{1}+1},~ s_2=\frac{m_{2}+1}{3m_{2}+2},~~s_3=\frac{nm_{3}+1}{(6n+1)m_{3}+6}.
\]

  Let $\mathcal{A}$ be a vertical annulus whose boundary is a vertical curve on $V_1$ side and the Legendrian regular fiber $L$. Then the dividing set of $\mathcal{A}$ will have boundary parallel arcs (hence bypasses) on $-\partial(M\setminus V_1)$ side whenever $2m_1+1<-t$. Those bypasses will potentially increase the twisting number $m_1$. More precisely, as the ruling slopes on $\partial V_1$ is $-\frac{1}{2}$, by the Twist Number Lemma from \cite{H000} and the choice of $L$, we can increase $m_1$ by one till either $2m_1+1=-t$  or $m_1=-1$. Similarly $m_2$ can be increased till either $3m_1+2=-t$ or $m_2=-1$. In particular, there is a non-negative integer $k$ satisfying $m_1=-3k-1$, $m_2=-2k-1$ and $t=6k+1$. Let $\mathcal{A}$ be a vertical annulus between $V_1$ and $V_2$ whose boundary now consists of vertical curves on $V_1$ and $V_2$ side. Note that the dividing set of $A$ cannot have any boundary parallel arcs because the maximality of $-t$. We can cut along $\mathcal{A}$ and round the corners to get a smooth manifold $M\setminus (V_1\cup V_2\cup \mathcal{A})$ such that $\partial(M\setminus (V_1\cup V_2\cup \mathcal{A})$ is smoothly isotopic to $\partial(M\setminus V_3)$. Moreover, by the Edge Rounding Lemma from \cite{H000} we compute its slope as

\[
s(\Gamma_{\partial(M_{n}\setminus V_1 \cup V_2 \cup \mathcal{A})})=-\frac{m_1}{2m_{1}+1} + \frac{m_{2}+1}{3m_{2}+2} - \frac{1}{2m_{1}+1}=-\frac{k}{6k+1}.
\]    
 
When measured in $\partial V_3$

\[
s(\Gamma_{\partial V_3}) = A^{-1}_{3}(6k+1,k)=-n+k.
\]

We show that $k\geq n$ is impossible. To this end, suppose $k \geq n$ which then implies $t=6k+1\geq 6n+1$, and that there is a neighborhood $V'_3$ of $F_3$ such that slope$\Gamma_{\partial V'_3}=\infty$. When measured with respect to $-\partial (M_n\setminus V'_3)$ this slope becomes $\frac{1}{6}$ which contradicts to $t\geq 6n+1$ whenever $n\geq1$. Indeed, in this case a Legendrian ruling curve fiber on $-\partial(M_{n}\setminus V'_3)$ will have twisting $-6$. Therefore we obtain that  $tw(\xi)=-6k-1$ for some $k$ with $0\leq k\leq n-1$, proving the claim.  
\end{proof}

 Now we can decompose $M_{n}$ as $(M_n\setminus V'_3)\cup V'_3$ where $(M_n\setminus V'_3)$ is made of $V_1$, $V_2$ and a neighborhood of the annulus $A$.  Since $V_1$ and $V_2$ are the standard neighborhoods for each $k$, each carries a unique tight contact structure. Moreover the dividing set of $A$, which we determined that it is made of only horizontal arcs, uniquely determines a contact structure in the neighborhood of $A$. Therefore, there is unique tight contact structure on $M_n\setminus V'_3$ relative to its boundary $\partial (M_n\setminus V'_3)\cong\partial V'_3$. On the other hand, $V'_3$ has boundary slope $-n+k$, and by \cite[Theorem $4.16$]{H000}, it carries exactly $n-k$ tight contact structures relative to its boundary. Since $0\leq k \leq n-1$, we get that the total number of tight contact structures on $M_n$ is at most $\frac{n(n+1)}{2}$.   
\end{proof}

\subsection{Lower bound} Let $M_{\infty}$ denote the manifold obtained by zero surgery along the left handed trefoil. This small Seifert fibered manifold has also the structure of a torus bundle over the circle as

\[
M_{\infty}=T^2\times [0,1]/\sim_{A}
\]

where $(x, 1)=(Ax, 0)$ and $A$ is the diffeomorphism of $T^2$ given by
\[
A=\left( \begin{array}{cc}
0&-1\\ 1&1
\end{array} \right)
\]

Recall the homotopy class of an oriented tangent 2-plane field $\xi$ on a $3$--manifold $M$ with $c_1(\xi)$ torsion is determined by two invariants \cite[Theorem~$4.16$]{Go98}: the two dimensional invariant is spin$^c$-structure $\mathfrak{s}_{\xi}$ induced by $\xi$ and the three dimensional invariant is $\theta(\xi)\footnote{This also denoted as $d_3(\xi)$ in some other references. We will stick with Gompf's notation. The two are related by $4d_3(\xi)=\theta(\xi)$}\in \mathbb{Z}$ defined by
\[
\theta(\xi) = c_1^2(X,J) -3\sigma(X) -2\chi(X),
\]
for any almost-complex 4-manifold $(X,J)$ with $\partial X = M$ and such that $\xi$ is the field of complex tangencies $TM\cap J(TM)$. Here $\chi(X)$ is the Euler characteristic and $\sigma(X)$ is the signature. If $H_1(M)$ has no $2$ torsion (this is true for example for the manifold $M_{\infty}$), then $c_1(\xi)$ does determine the $\mathfrak{s}_{\xi}$.

The classification of tight contact structures on $M_{\infty}$, up to isotopy/homotopy and their fillability are well understood. The following theorem summarize those.

\begin{theorem}  $M_{\infty}$ admits an infinite family of tight contact structures $\{\xi_{i}\}^{\infty}_{i=0}$ such that
\begin{enumerate}
\item~\cite[Theorem~$0.1$]{H00}\cite{G99} They have the Giroux torsion $Tor(\xi_{i})=i$, and hence are pairwise non-isotopic. Moreover, any other tight structure on $M_{\infty}$ is isotopic to $\xi_i$ for some $i$.     
\item~\cite[Theorem~$0.1$]{H00}\cite{G99} All are homotopic with $\theta(\xi_{i})=4$ (note that $c_1(\xi_i)=0$ for all $i\geq 0$).
\item~\cite[Theorem~$0.1$]{H00}\cite{G99} All are weakly fillable.
\item~\cite{GAY06} The contact structure $\xi_0$ is Stein (and hence strongly) fillable, while for $i>0$, $\xi_i$ is not strongly fillable. 
\item~\cite[Theorem~$1$]{GHM07}. The contact invariants $c(\xi_{i})$ which have degree $-\frac{\theta(\xi_i)}{4}-\frac{1}{2}$ with untwisted coefficients are all zero for $i>0$, while with twisted coefficients all are non-zero and pairwise different for $i\geq 0$.    
\end{enumerate}
\end{theorem}

Since the contact structures $\xi_i$ are all vertical, the knot

\[
F'={\bf 0}\times [0,1]/\sim_{A}
\] 

is tangent to $\xi_i$ for all $i$. If we use the Seifert fibration structure on $M_{\infty}$, i.e. the surgery description in Figure~\ref{Figure4}, then the knot $F'$ is topologically isotopic to a meridian of the left handed trefoil. The manifold $M_{\infty}$ can also be described as the boundary of $E_9$ plumbing from which we obtain its the unique Stein fillable contact structure. If $F$denotes the image of $F'$ under this isotopy and $F'$ has framing $f$, then $F$ has framing $f+1$.

\begin{figure}[!h]
\centering
\includegraphics[scale=0.7]{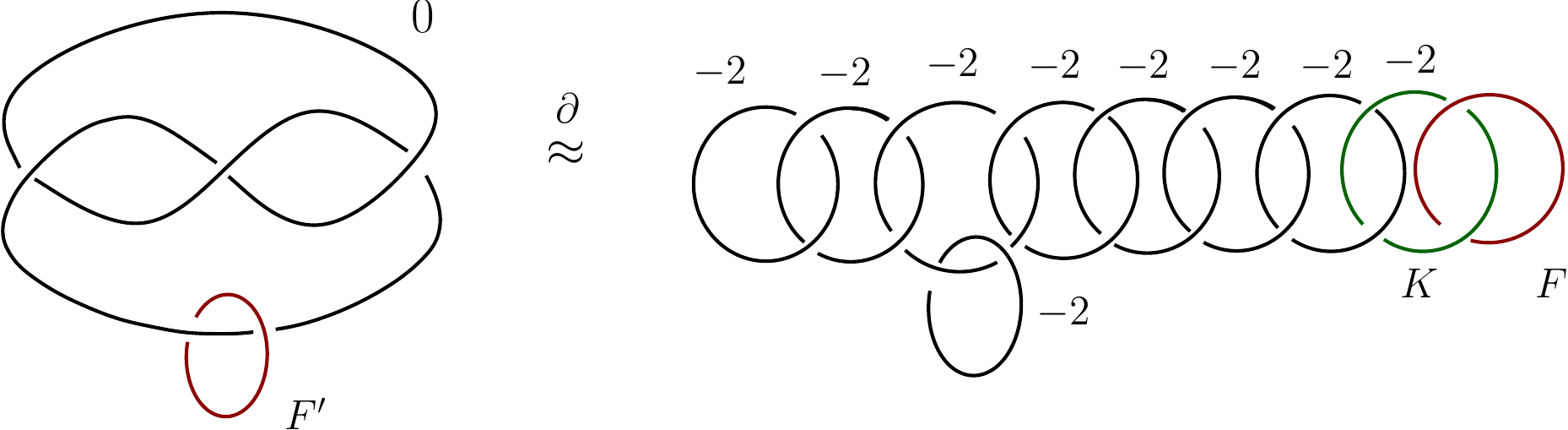} 
\caption{}
\label{Figure4}
\end{figure}

\begin{lemma}\cite[Lemma~$3.5$]{G06}
There exists a framing for $F$ such that $tw(F,\xi_{i})=-i-1$. Moreover smooth $-n-1$--surgery along $F$ gives $M_{n}$.
\end{lemma}

%\begin{proof}
%This is essentially Lemma $3.5$ in \cite{G06} where Ghiggini proved this fact for $\overline{M_{\infty}}$. In particular, If we do smooth $-n$--surgery along $F'$, it is equivalent to do smooth $-n-1$ surgery along $F$.     
%\end{proof}
 
Let $F_{i,j}$ denote the Legendrian knot obtained from $F$ by applying $n-i-1$ stabilizations such that the rotation number $\rot(F_{i,j},\xi_i)=j$ where $0\leq i\leq n-1$ and $|j|\leq n-i-1$ with $j\equiv n+1-i (\textrm{mod}~2)$. Finally, let $\xi^n_{i,j}$ denote tight contact structures on $M_{n}$ obtained by Legendrian surgery along $F_{i,j}$ in $(M_{\infty},\xi_{i})$.

\begin{lemma}\label{conclasscal}
The contact structures $\xi^n_{i,j}$ are pairwise non-isotopic.  
\end{lemma}

 Note that because Legendrian surgery preserves fillability, we obtain that all $\xi^n_{i,j}$ are weakly fillable and hence tight. Among them $\xi^n_{0,j}$, see Figure~\ref{FigureF}, are Stein fillable as $\xi_{0}$ is Stein fillable with $\theta(\xi^n_{0,j})=2$. All other tight structures $\xi^n_{i,j}$for $i>0$ can be made at least strongly fillable because $M_n$ is an integral homology sphere. By work of Lisca-Matic \cite{LM97}, we have that for any $n\geq 1$, the tight contact structures $\xi^n_{0,j}$ are pairwise non-isotopic. Moreover by work of Plamenevskaya ~\cite[Section~$3$]{P04} the contact invariants $c(\xi^n_{0,j})\in \widehat{HF}_{(-1)}(-M_{n})$ are linearly independent over $\mathbb{Z}$. Since, $\widehat{HF}_{(-1)}(-M_{n})\cong \mathbb{Z}^n$ ~\cite[Section~$8.1$]{OZ03}, we obtain that $\widehat{HF}_{(-1)}(-M_{n}) \cong \left\langle  c(\xi^n_{0,-n+1}),c(\xi^n_{0,-n+3}), \cdots, c(\xi^n_{0,n-1}) \right\rangle$.  We shall prove now that all other contact classes $c(\xi^n_{i,j})$ for $i>0$ are distinct linear combinations of these, from which Lemma~\ref{conclasscal} follows immediately. Moreover as the degree of the each of these contact classes is $-\frac{\theta}{4}-\frac{1}{2}=-1$ and $M_n$ is an integral homology sphere,  we also conclude that the contact structures $\xi^n_{i,j}$ all are homotopic with $\theta=2$. 

\begin{figure}[!h]
\centering
\includegraphics[scale=0.7]{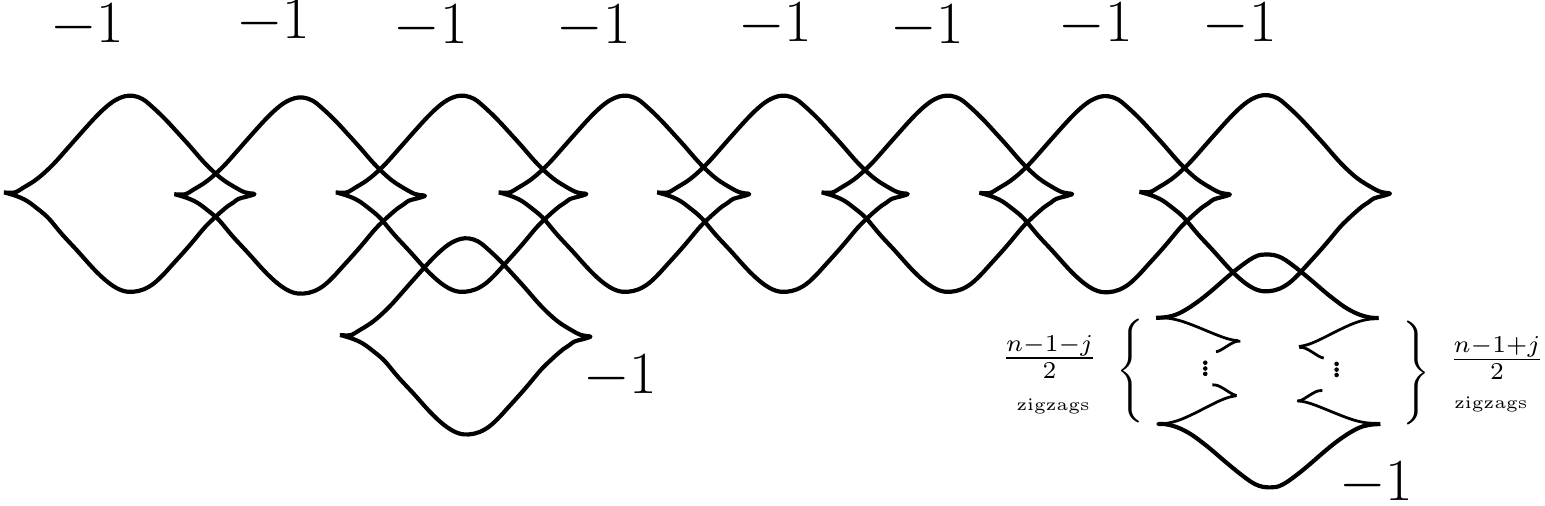} 
\caption{Stein fillable tight contact structures $\xi^n_{0,j}$ on $M_{n}$.}
\label{FigureF}
\end{figure}

\begin{claim}\label{cal}
The contact invariant of $\xi^n_{i,j}$ is 
$$c(\xi^n_{i,j})=\sum^i_{k=0}(-1)^k {i \choose k}c(\xi^n_{0,j-i+2k}).$$

In particular the contact structures $\xi^n_{i,j}$ have non-vanishing and pairwise different contact classes for $(i,j)$ with $0\leq i\leq n-1$ and $|j|\leq n-i-1$ with $j\equiv n+1-i (\textrm{mod}~2)$.
\end{claim}

\begin{proof}

{\bf Step--I:} We describe a sequence of Legendrian surgeries and hence a cobordism $Z_{n}$ between $M_{\infty}$ and $M_n$. This is essentially encoded in the genus one open book decomposition in Figure~\ref{Figure5} as follows. First, note that if we let $l=r=0$ and ignore the red curve we obtain an open book decomposition for $M_{\infty}, \xi_{i}$~\cite{VHM07}. When $i=0$, the open book is compatible with the unique Stein fillable contact structure $\xi_{0}$. This much is a slight modification of the picture in Van Horn-Morris thesis~\cite[Theorem~$4.3.1.d$]{VHM07} (see also \cite[Table~$2$]{E06}). Adding the segment of words as in the lower part of the open book {\it i times} corresponds to increasing the Giroux torsion by $i$. The curve $L_{l,r}$ is the Legendrian knot $F_{i,j}$ in $M_{\infty}, \xi_{i}$ where $l$ and $r$ stand for the number of the positive and negative stabilizations, respectively (see \cite{GM09}) and satisfy $j=l-r$ and $n=l+r+i+1$. If we perform Legendrian surgery along $F_{i,j}$, that is add a positive Dehn twist to the open book along $L_{l,r}$, we obtain an open book decomposition compatible with $M_n, \xi^n_{i,j}$.

\begin{figure}[!h]
\centering
\includegraphics[scale=0.7]{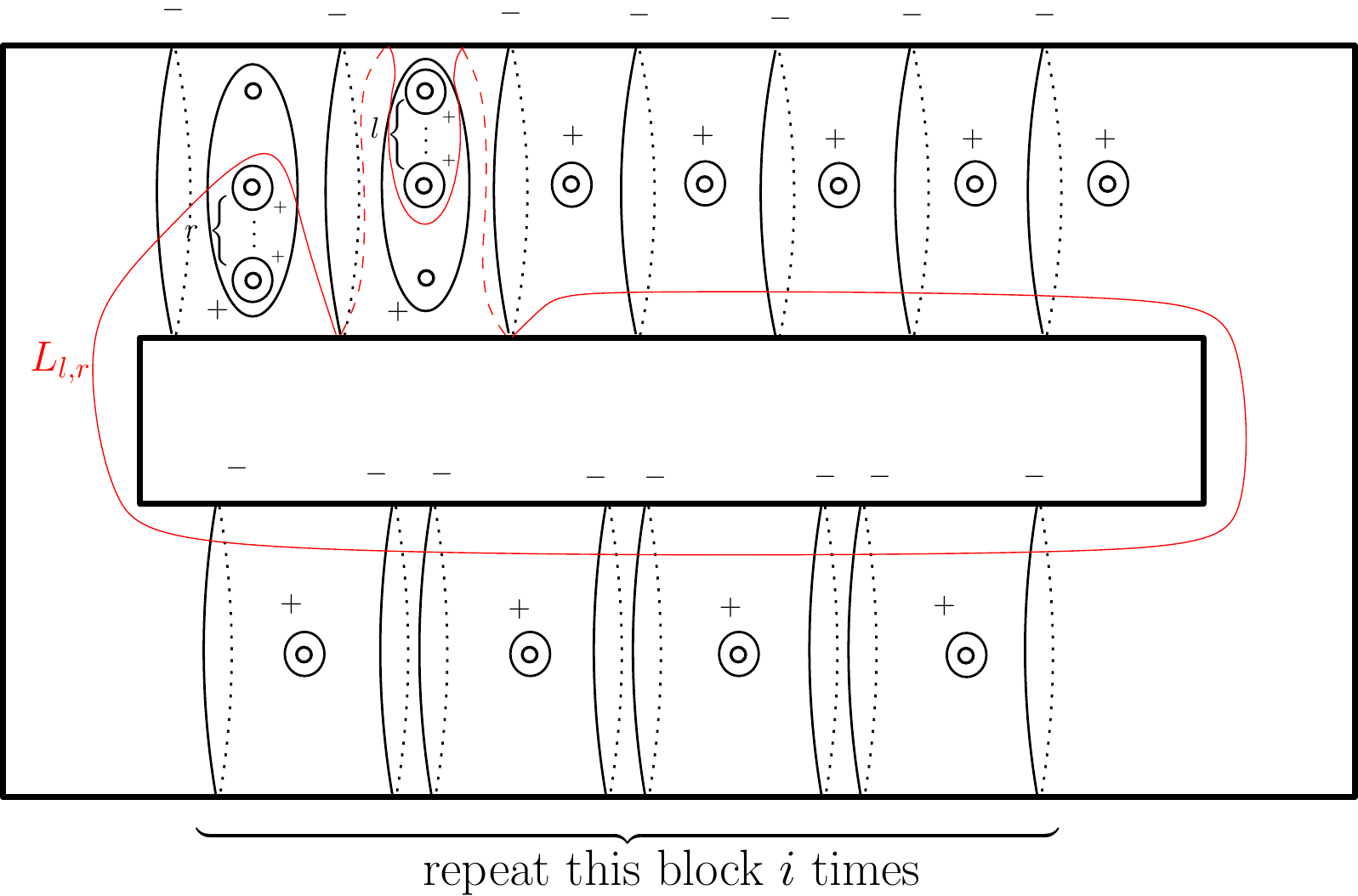} 
\caption{A genus one open book decomposition of $(M_{\infty},\xi_i)$ and $(M_n,\xi^n_{i,j})$. Here the thick circles are boundary components. To obtain an open book supporting $(M_{\infty},\xi_i)$, take the monodromy as the composition of Dehn twists, with indicated signs, along all curves except $L_{l,r}$. The bottom block of curves repeats $i$ times, and meant to introduce the Giroux torsion. The Legendrian curve $L_{l,r}$ is a realization of $F_{i,j}$ on the page, where $l$ and $r$ stand for the number of the positive and negative stabilizations of $F$, respectively and satisfy $j=l-r$ and $n=l+r+i+1$. To obtain an open book supporting $(M_n,\xi^n_{i,j})$ add a positive Dehn twist along $L_{l,r}$ in $(M_{\infty},\xi_i)$. }
\label{Figure5}
\end{figure}

Second, we introduce exactly $12$ simple closed curves, denote this link by $C$, and a positive Hopf stabilization in the open book compatible with $\xi_{i+1}$ as in Figure~\ref{Figure6}-(a). These curves can be Legendrian realized on the page. If we perform Legendrian surgeries along those, that is add positive Dehn twists to the open book along the blue curves, we get the open book in Figure ~\ref{Figure6}-(b). Finally, we apply the braid relation~\cite[Lemma~$4.4.2$]{VHM07} four times to get the open book in Figure~\ref{Figure6}-(c) which is an open book for $(M_{\infty}, \xi_{i})$, that is via this procedure we reduced the Giroux torsion by $1$.

\begin{figure}[!h]
\centering
\includegraphics[scale=0.7]{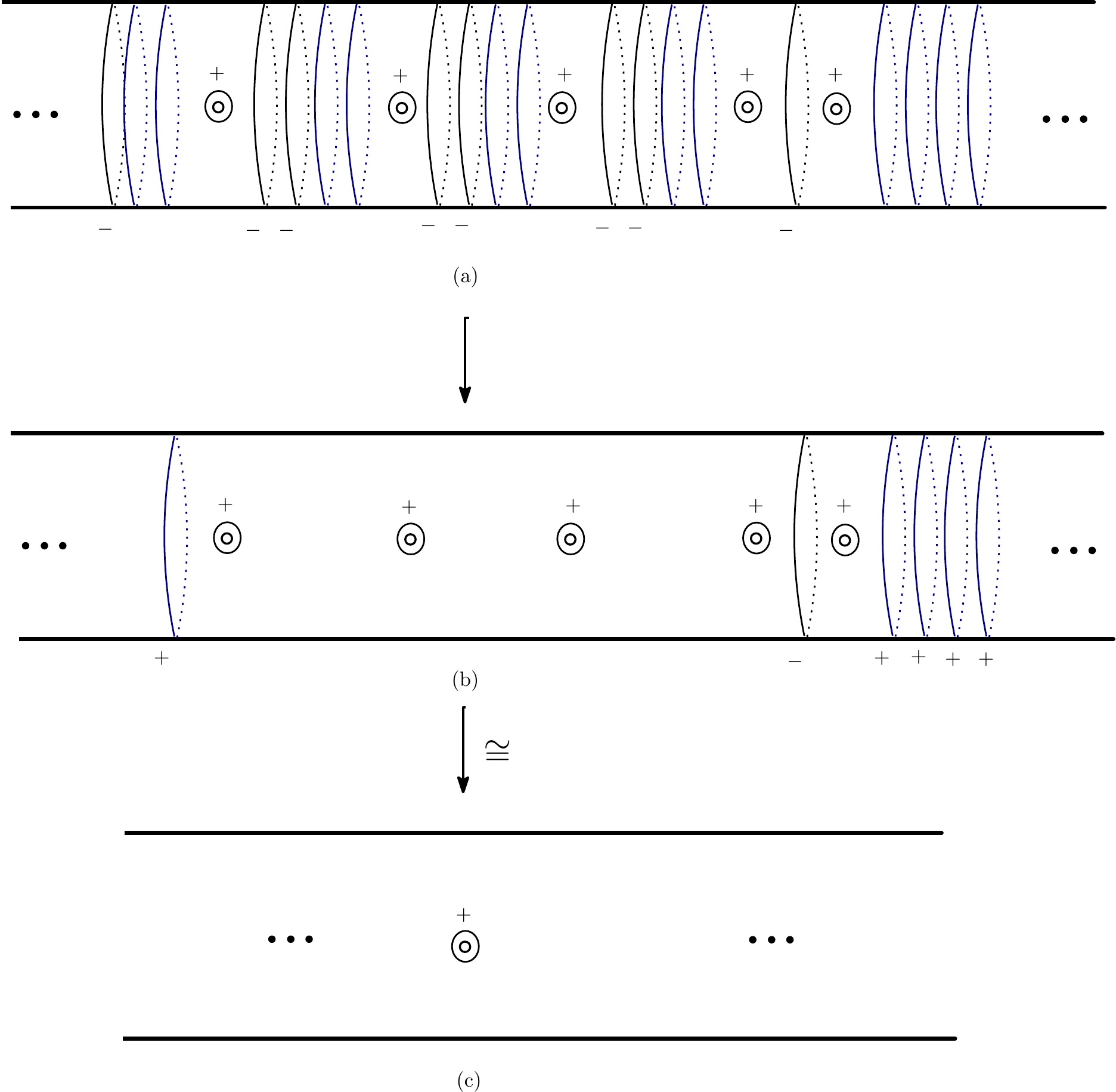} 
\caption{In $(a)$, we see a piece of open book for $(M_{\infty},\xi_{i+1})$ describing a Giroux torsion one plus a positive Hopf stabilization. In this stage we introduce exactly $12$ simple closed curves, shown in blue, which make up the link $C$. In $(b)$, we see the result of adding positive Dehn twists along $C$ to the monodromy. After applying the braid relations (four times) to this open book, we reach the open book in $(c)$ which is an open book for $(M_{\infty},\xi_i)$.}
\label{Figure6}
\end{figure}

Now, let $Z_n$ denote the cobordism between $(M_{\infty}, \xi_{i+1})$ and $(M_n,\xi^n_{i,j}$ which is obtained by attaching Stein $2$--handles along $\L_{l,r}\cup C$. This cobordism can be decomposed in two ways:

\begin{itemize}
\item We attach $2$--handles along $C$, first, and then along $\L_{l,r}$, to obtain a cobordism $W_{\infty}$ from $M_{\infty}$ to itself, followed by a cobordism $V_{n}$ from $M_{\infty}$ to $M_n$.   

\item We attach $2$--handles along $L_{l,r}$, first, and then along $C$, to obtain a cobordism $V_{n+1}$ from $M_{\infty}$ to $M_{n+1}$, followed by a cobordism  $W_n$ from $M_{n+1}$ to $M_n$.   
\end{itemize}

{\bf Step--II:} These cobordisms induce maps on Heegaard Floer homology which fit in the following diagram.

\[\begin{diagram}\label{diag}
&\widehat{HF}(-M_n) & \rTo^{F_{W_n,\mathfrak s}} & \widehat{HF}(-M_{n+1})&\\
&\dTo^{F_{V_n,\mathfrak s}} &\hspace*{2cm}& \dTo_{F_{V_{n+1},\mathfrak s}}\\
& \widehat{\underline{HF}}(-M_{\infty}) & \rTo_{F_{W_{\infty}, \mathfrak s}} & \widehat{\underline{HF}}(-M_{\infty}),
\end{diagram}
\]

where $\widehat{\underline{HF}}(M)$ denotes the Heegaard Floer homology of $M$ with twisted coefficients. It has a structure of $\mathbb{Z}[H^1(M)]$-module. Much of Heegaard Floer theory with usual coefficients extends to twisted one. The details of this can be found in ~\cite{OZ04}, \cite{JM08}, and \cite{GM09}. The last reference is particularly helpful for its properties relevant to contact geometry. As explained in \cite{GM09}, main reason using Heegaard Floer with twisted coefficients is that with twisted coefficients the vertical maps are injective in the relevant degrees.

Let $\Lambda=\mathbb{Z}[H^1(M_{\infty},\frac{1}{2}\mathbb{Z})]$ denote the coefficient ring. The introduction of $\frac{1}{2}\mathbb{Z}$ is for symmetry of formulas. Let $\mathfrak s$ denotes the canonical Spin$^c$ structure on cobordisms.  We point out some features of Diagram~\ref{diag} as follows.

\begin{lemma} 
\begin{enumerate}
\item \label{comp} $\widehat{\underline{HF}}(-M_{\infty}, \Lambda)=\Lambda_{-\frac{1}{2}}\oplus\Lambda_{-\frac{3}{2}}$. Moreover, the contact invariant of Stein fillable structure $c(\xi_0)$ is identified with $1$ in summand in degree $-\frac{3}{2}$. 
\item \label{lemma412} We can make choices so that $F_{V_n,\mathfrak s}(c(\xi^n_{0,j}))=t^{j/2}$. 
\item \label{lemma413} If we further choose $F_{W_{\infty},\mathfrak s}$ to be represented by $t^{\frac{1}{2}}-t^{-\frac{1}{2}}$, then Diagram~\ref{diag} commutes.
\end{enumerate}
\end{lemma}
 
\begin{proof}[Proof of \ref{comp}] The corresponding computation for $M_{\infty}$ was done in \cite{GM09} from which our computation follows by reversing orientation. Since the Stein fillable contact structure $\xi_0$ has $\theta(\xi_0)=4$, we obtain that its contact class $c(\xi_0)$ has degree $-\frac{\theta(\xi_0)}{4}-\frac{1}{2}=-\frac{3}{2}$. So $c(\xi_0)$ is the generator of $\Lambda_{-\frac{3}{2}}$. 
\end{proof}

\begin{proof}[Proof of \ref{lemma412}] The proof of this fact is very similar to Lemma~$4.12$ from \cite{GM09}. We only point the main difference. In Figure~\ref{Figure4}-(b), if we ignore the knots $F$, and $K$, then the remaining surgery is boundary diffeomorphic to the Poincar\'{e} homology sphere $\Sigma(2,3,5)$. Let $V_{\infty}$ denote the cobordism between $\Sigma(2,3,5)$ and $M_{\infty}$ obtained by attaching a $2$--handle along the knot $K$ with framing -2. The manifold $\Sigma(2,3,5)$ carries a unique tight contact structure which is Stein fillable. This well known fact can be obtained for example from Lemma~\ref{guidelemma} as well. Let $\xi_{std}$ denote this tight structure. Since $\Sigma(2,3,5)$ is an integral homology sphere one can talk about the Thurston-Bennequin number $\tb$ of a Legendrian knot. In particular, it is easy to see the knot $K$ has a Legendrian representative with $\tb=-1$. So, the $2$--handle can be attached to $K$ in a Stein way, that is $V_{\infty}$ is a Stein cobordism. In particular if $\mathfrak t$ denote the canonical Spin$^c$ structure coming from the Stein structure, then $[F_{V_{\infty},\mathfrak t}(c(\xi_0))]=[c(\xi_{std})]$. Now if we identify $\widehat{\underline{HF}}_{(-\frac{3}{2})}(-M_{\infty})$ and $\widehat{\underline{HF}}(\Sigma(2,3,5))_{(-2)}[t^{1/2},t^{-1/2}]$ with $\mathbb{Z}[t^{\pm 1/2}]$, then since  the map $F_{V_{\infty},\mathfrak t}$ is an isomorphism that is defined up to multiplication by a power of $t$, we can choose it to be the conjugation map. The rest is identical to Lemma~$4.12$ in \cite{GM09}.   
\end{proof}

\begin{proof}[Proof of \ref{lemma413}] This is identical to Lemma~$4.13$ from \cite{GM09}.     
\end{proof}

{\bf Step--III:} All in place, we can prove Claim~\ref{cal} by induction on $n$. For the base case $n=1$, there is nothing to prove as $M_1$ admits a unique tight contact structure by Theorem~\ref{main}-(c). Assume the formula in Claim~\ref{cal} holds for the tight contact structures on $M_n$ for some $n$. We prove it for the ones on $M_{n+1}$.  From the Stein cobordism explained above, we have   ${F_{W_n,\mathfrak s}}(c(\xi^n_{i,j}))=c(\xi^{n+1}_{i+1,j})$, and, by the induction hypothesis: 
\[
c(\xi^{n+1}_{i+1,j})=\sum^i_{k=0}(-1)^k {i \choose k}c(\xi^{n+1}_{1,j-i+2k})
\] 

By the commutativity of Diagram~\ref{diag} and the injectivity of the map ${F_{V_{n+1},\mathfrak s}}$, we find that $c(\xi^{n+1}_{1,j})=c(\xi^{n+1}_{0,j+1})-c(\xi^{n+1}_{0,j-1})$. Finally substituting this in the sum above, and simple modifications o indices gives that 
\[
c(\xi^{n+1}_{i+1,j})=\sum^{i+1}_{k=0}(-1)^{k-1} {i+1 \choose k}c(\xi^{n+1}_{1,j-(i+1)+2k})
\]

which finishes proof of Claim~\ref{cal}.
\end{proof}

\subsection{Strongly fillable but not Stein fillable structures.}

\begin{lemma}
The manifold $M_{n}$ for $n\geq 1$ admits at least $\left\lfloor\frac{n}{2} \right\rfloor$ tight contact structures which are strongly fillable but not Stein fillable. 
\end{lemma}

\begin{proof}
We claim that the tight contact structures $\xi^n_{i,0}$ are not Stein fillable for $0<i\leq n-1$ with $i\equiv n-1(mod~2)$. To this end, suppose $(X_i,J_i)$ is a Stein filling of $(M_{n},\xi^n_{i,0})$. By Claim~\ref{cal}, we can express the contact classes $c(\xi^n_{i,0})$ as
\begin{equation}\label{form}
c(\xi^n_{i,0})=\sum^i_{k=0}(-1)^k {i \choose k}c(\xi^n_{0,-i+2k})
\end{equation}

Moreover, if $\mathcal J$ denotes the conjugation map in Heegaard Floer homology, then it is easy to see from the surgery description in Figure~\ref{FigureF} that the contact structure $\overline{\xi^n_{0,j}}=\mathcal J(\xi^n_{0,j})$ is isotopic to $\xi^n_{0,-j}$, and $\xi^n_{i,0}$ is isotopic to its conjugate for each $0\leq i \leq n-1$ with $i\equiv n-1~(\textrm{mod}~2)$. We also point out that the coefficients of the contact classes $c(\xi^n_{0,j})$ and $c(\xi^n_{0,-j})$ in Formula~\ref{form} are the same. Therefore, we conclude that $c(\xi^n_{i,0})$ is the linear combination of elements of the form $c(\xi^n_{0,j})+c(\xi^n_{0,-j})$ for $j=-n+1,-n+3,\cdots, n-1$.     

We can puncture the Stein filling $(X_i,J_i)$ and view it as a Stein cobordism from $-M_{n}$ to $S^3$. Note that if $\mathfrak t_i$ denotes the canonical Spin$^c$ structure coming from the Stein structure $J_i$, then $\mathfrak t_i$ is isomorphic to its conjugate because $\xi^n_{i,0}\cong \overline{\xi^n_{i,0}}$. In particular, by \cite[Theorem~4]{P04} $F^{+}_{X_i,\mathfrak t_i}(c(\xi^n_{i,0}))$ is a generator of $HF^+(S^3)\cong \mathbb{F}$, where $\mathbb{F}\cong \mathbb{Z}_2$ is the coefficient ring. On the other hand by using the fact that the conjugation homomorphism $\mathcal J$ has trivial action on $HF^+(S^3)$ and observations above, we conclude that the map $F^{+}_{X_i,\mathfrak t_i}$ evaluates on each $c(\xi^n_{0,j})+c(\xi^n_{0,-j})$ as        

\[
F^{+}_{X_i,\mathfrak t_i}(c(\xi^n_{0,j})+c(\xi^n_{0,-j}))=2F^{+}_{X_i,\mathfrak t_i}(c(\xi^n_{0,j}))=0.
\]

Note that if $n$ is odd, then the linear combination will include $c(\xi^n_{0,0})$ where $\xi^n_{0,0}$ is a Stein fillable contact structures which is isotopic to its conjugate. We don't know if $F^{+}_{X_i,\mathfrak t_i}(c(\xi^n_{0,0}))=0$, but we do have a crucial detail that the contact class $c(\xi^n_{0,0})$ has {\it even} coefficient ${(n-i) \choose \frac{(n-i)}{2}}$, for $1 \leq i \leq n-2$. In particular, for either parity of $n$,  we end up with the contradiction that $F^{+}_{X_i,\mathfrak t_i}(c(\xi^n_{i,0}))=0$.  Thus, $(M_{n},\xi^n_{i,0})$ cannot be Stein fillable for any $0<i\leq n-1$ with $i\equiv n-1(\textrm{mod}~2)$. This counts to the total of $\left\lfloor\frac{n}{2} \right\rfloor$ strongly fillable tight structures on $M_{\infty}$ that are not Stein fillable.

\end{proof}
\bibliographystyle{amsplain}

\begin{thebibliography}{10}

%\bibitem{BakerEtnyreVanHornMorris10}
%Kenneth L.\ Baker, John~B. Etnyre and Jeremy {Van Horn-Morris}.
%\newblock Cabling, rational open book decompositions and contact structures.
%\newblock {\em J. Differential Geometry}, $\mathbf{90}$  (2012), 1--80.



%\bibitem{B83} 
% D. Bennequin.
%\newblock Entralacements et \'{e}quations de Pfaff.
%\newblock {\em Ast\'{e}risque} $\mathbf{107-108}$ (1983) 87--161


%\bibitem{C97}
% V. Colin.
%\newblock Chirurgies d'indice un et isotopies de sph\`{e}res dans les
%  vari\'et\'es de contact tendues.
%\newblock {\em C. R. Acad. Sci. Paris S\'er. I Math.}, $\mathbf{324(6)}$, 1997, 659--663.


\bibitem{CGH09}
 V. Colin, E. Giroux, K. Honda
\newblock Finitude homotopique et isotopique des structures de contact tendues.
\newblock {\em  Publ. Math. Inst. Hautes \'{E}tudes Sci.}, $\mathbf{109}$, (2009), 245--293.


\bibitem{E06} 
 T. Etgu.
\newblock Elliptic open books on torus bundles over the circle
\newblock {\em Geom. Dedicata}, $\mathbf{132}$, (1990), 53-63.

\bibitem{EH01}
 John~B. Etnyre and K. Honda.
\newblock On the nonexistence of tight contact structures.
\newblock {\em Ann. of Math.} $\mathbf{153(3)}$ (2001), 749--766.


\bibitem{GAY06}
 D.T. Gay.
\newblock  Four-dimensional symplectic cobordisms containing three-handles.
\newblock {\em Geom. Topol.} $\mathbf{10}$ (2006), 1749--1759. 


\bibitem{G06}
 P. Ghiggini.
\newblock   Ozsv\'{a}th–Szab\'{o} invariants and fillability of contact structures.
\newblock {\em Math. Z} $\mathbf{253}(1)$ (2006), 159--175.

\bibitem{G08}
 P. Ghiggini.
\newblock  Strongly fillable contact 3-manifolds without Stein fillings.
\newblock {\em Geom. Topol.} $\mathbf{9}$ (2005), 1677--1687.

\bibitem{G008}
 P. Ghiggini.
\newblock On tight contact structures with negative maximal twisting number on small Seifert manifolds.
\newblock {\em Algebr. Geom. Topol.} $\mathbf{8(1)}$ (2008), 381--396. 


\bibitem{GLS06}
 P. Ghiggini, P.Lisca, A. Stipsicz.
\newblock Classification of tight contact structures on small Seifert 3-manifolds with $e_0\geq=0$.
\newblock {\em  Proc. Amer. Math. Soc.} $\mathbf{134(3)}$ (2006), 909--916. 
 

\bibitem{GLS07}
 P. Ghiggini, P.Lisca, A. Stipsicz.
\newblock Tight contact structures on some small Seifert fibered 3-manifolds.
\newblock {\em Amer. J. Math.} $\mathbf{129(5)}$ (2007), 1403--1447. 


\bibitem{GHM07}
 P. Ghiggini, K. Honda, J. Van Horn-Morris.
\newblock The vanishing of the contact invariant in the presence of torsion.
\newblock arXiv:0706.1602. (2007). 


\bibitem{GS}
 P. Ghiggini, S. Sc\"{o}nenberger.
\newblock On the classification of tight contact structures.
\newblock In {\em Topology} and {\em Geometry of Manifolds of Proceedings of Symposia in Pure Mathematics Amer. Math. Soc.}. $\mathbf{71}$ (2009), 121--151. 
 
\bibitem{GM09}
 P. Ghiggini, J. Van Horn-Morris.
\newblock Tight contact structures on the Brieskorn spheres $-\Sigma(2,3,6n-1)$ and contact invariants.
\newblock arXiv:0910.27522. (2009).  
 
\bibitem{G91}
 E. Giroux.
\newblock Convexit\'{e} en topologie de contact.
\newblock {\em Comment. Math. Helv.}, $\mathbf{66}$ (1991),637-677.


\bibitem{G00}
 E. Giroux.
\newblock Structures de contact en dimension trois et bifurcations des feuilletages de surfaces.
\newblock {\em Invent. Math.} $\mathbf{141}$(2000),615-689.


\bibitem{G99}
 E. Giroux.
\newblock Une infinit\'{e} de structures de contact tendues sur une infinit\'{e} de vari\'{e}t\'{e}s.
\newblock {\em Invent. Math.} $\mathbf{135}(3)$(1999),789-802.


\bibitem{Go98}
 R. Gompf.
\newblock Handlebody construction of Stein surfaces.
\newblock In {\em Ann. Math.}  $\mathbf{148}$ (1998), 619--693.

\bibitem{Hatcher}
 A. Hatcher.
\newblock Notes on basic of $3$--manifold topology.

\bibitem{H000}
 K. Honda.
\newblock On the classification of tight contact structures~I
\newblock {\em Geometry \& Topology}, $\mathbf{4}$ (2000),309-368. \textit{Factoring nonrotative $T^2\times I$ layers}, Erratum to  ``On the classification of tight contact structures I", {\em Geometry \& Topology}, $\mathbf{5}$ (2001), 925-938

\bibitem{H00}
 K. Honda.
\newblock On the classification of tight contact structures~II.
\newblock {\em J. Differential Geom.} $\mathbf{55}$ (2000),83-143.

\bibitem{HKM2007}
 K. Honda W. Kazez and G. Matic .
\newblock Right-veering diffeomorphisms of compact surfaces with boundary.
\newblock {\em Invent. math.} $\mathbf{169}$ (2007),427-449.


\bibitem{IK2014}
 T. Ito and K. Kawamuro .
\newblock Visualizing overtwisted discs in open books.
\newblock {\em  Publ. Res. Inst. Math. Sci.} $\mathbf{50(1)}$ (2014),169--180. 

\bibitem{KR2012}
 W. Kazez and R. Roberts .
\newblock Fractional Dehn twists in knot theory and contact topology.
\newblock {\em Algebr.Geom. Topol.} $\mathbf{13(6)}$ (2013),3603--3637. 
 
%\bibitem{H02}
% K. Honda.
%\newblock Gluing tight contact structures.
%\newblock {\em Duke Math. J.} $\mathbf{115}$ (2002), 435-478.

%\bibitem{Kanda}
% Y. Kanda.
%\newblock The classification of tight contact structures on the $3$--torus.
%\newblock {\em Comm. Anal. Geom.} $\mathbf{5}$ (1997), 413-438.



\bibitem{L2010}
 Yank\i\ Lekili.
\newblock Planar open books with four binding components.
\newblock {\em Algebr. Geom. Topol.} $\mathbf{11}$ (2011), 909--928.



\bibitem{L2012}
 P.Lisca.
\newblock On overtwisted, right-veering open books.
\newblock {\em Pac. J. Math.} $\mathbf{257}$ (2012), 219--225. 


\bibitem{LS07}
 P.Lisca, A. Stipsicz.
\newblock  Ozsv\'{a}th-Szab\'{o} invariants and tight contact $3$-manifolds. III.
\newblock {\em J. Symp. Geom.} $\mathbf{5(4)}$ (2007), 357--384.

\bibitem{LS09}
 P.Lisca, A. Stipsicz.
\newblock On the existence of tight contact structures on Seifert fibered 3-manifolds.
\newblock {\em Duke Math J.} $\mathbf{148(2)}$ (2009), 175--209. 
 

\bibitem{LM97}
 P.Lisca, G. Mati\'{c}.
\newblock Tight contact structures and Seiberg-Witten invariants.
\newblock {\em Invent. Math.} $\mathbf{129(3)}$ (1997), 509--525. 


\bibitem{JM08}
 S. Jabuka and T. Mark
\newblock Product formulae for Ozsváth–Szabó 4-manifold invariants.
\newblock {\em Geom. Topol} $\mathbf{12}(3)$ (2008), 1557--1651. 

\bibitem{M2008}
 P.Massot.
\newblock Geodesible contact structures on $3$--manifolds.
\newblock {\em Geom. Topol.} $\mathbf{12}$ (2008), 1729--1776.

\bibitem{P04}
 O. Plamenevskaya.
\newblock Contact structures with distinct Heegaard Floer invariants.
\newblock {\em Math. Res. Lett.} $\mathbf{11}(4)$ (2004), 547–561.

\bibitem {OZ03}
P. Ozsv\'{a}th and Z. Szab\'{o}.
\newblock Absolutely graded Floer homologies and intersection forms for four-manifolds with boundary.
\newblock {\em Adv. of Math.} $\mathbf{173(2)}$ (2003), 179–261.

\bibitem {OZ04}
P. Ozsv\'{a}th and Z. Szab\'{o}.
\newblock Holomorphic disks and three-manifold invariants: Properties and applications.
\newblock {\em Ann. of Math.} $\mathbf{159(3)}$ (2004), 1159–1245.


\bibitem{VHM07}
J. Van Horn-Morris.
\newblock  Constructions of open book decompositions.
\newblock Ph.D. dissertation, University of Texas at Austin, (2007)

\bibitem{W06}
 H. Wu.
\newblock Legendrian vertical circles in small Seifert spaces.
\newblock {\em Commun. Contemp. Math.} $\mathbf{8(2)}$ (2006), 219--246.  



%\bibitem{ef} 
%Y. Eliashberg and M. Fraser.
%\newblock Classification of topologically trivial Legendrian knots.
%\newblock {\em In Geometry,topology, and dynamics} (Montreal, PQ, 1995), pages 17�51, {\em CRM Proc. Lecture Notes}, 15, 1998.

%\bibitem{e}  
%John~B. Etnyre.
%\newblock Transversal torus knots.
%\newblock {\em Geometry \& Topology}  $\mathbf{3}$ (1999), 253-268.

%\bibitem{e1} 
% John~B. Etnyre.
%\newblock Legendrian and transversal knots, in the {\em Handbook of knot theory} (Elsevier B. V., Amsterdam), 2005, 105-185. 
 
%\bibitem{eh1}
% John~B. Etnyre and K. Honda.
%\newblock Cabling and Transverse simplicity.
%\newblock {\em Ann.of Math.} (2) $\mathbf{162}$ (2005), 1305-1333.
 
\end{thebibliography}

\end{document}